\documentclass[a4paper,10pt,reqno]{amsart}
\usepackage{amssymb,amsmath,mathrsfs,graphics,graphicx,mathptm,float,lscape}
\usepackage[T1]{fontenc}
\usepackage[english]{babel}
\usepackage[all]{xy}
\usepackage{lscape}
\usepackage{pgf,tikz}
\usetikzlibrary{arrows}

\theoremstyle{plain}
\newtheorem{thm}{Theorem}[section]
\newtheorem{pro}[thm]{Proposition}
\newtheorem{lem}[thm]{Lemma}
\newtheorem{cor}[thm]{Corollary}

\newtheorem{question}{Question}

\theoremstyle{definition}
\newtheorem*{defi}{Definition}

\newtheorem{eg}[thm]{Example}
\newtheorem{egs}[thm]{Examples}
\newtheorem{rem}[thm]{Remark}
\newtheorem{rems}[thm]{Remarks}

\usepackage[colorlinks, linktocpage, citecolor = blue, linkcolor = blue]{hyperref}

\def\og{\leavevmode\raise.3ex\hbox{$\scriptscriptstyle\langle\!\langle$~}}
\def\fg{\leavevmode\raise.3ex\hbox{~$\!\scriptscriptstyle\,\rangle\!\rangle$}}

\addtocounter{section}{0}             % Start with section 1
\numberwithin{equation}{section}

\begin{document}
\selectlanguage{english}

\title[Birational maps preserving the contact structure on $\mathbb{P}^3_\mathbb{C}$]{Birational maps preserving \\the contact structure on $\mathbb{P}^3_\mathbb{C}$}

\author{Dominique Cerveau}

\address{IRMAR, UMR 6625 du CNRS, Universit\'e de Rennes $1$, $35042$ Rennes, France.}
\email{dominique.cerveau@univ-rennes1.fr}

\author{Julie D\'eserti}

\address{Universit\'e Paris Diderot, Sorbonne Paris Cit\'e, Institut de Math\'emtiques de
Jussieu-Paris Rive Gauche, UMR $7586$, CNRS, Sorbonne Universit\'es, UPMC Univ Pris $06$,
F-$75013$ Paris, France.}

\email{deserti@math.univ-paris-diderot.fr}

\begin{abstract}
We study the group of polynomial automorphisms of $\mathbb{C}^3$ (resp. birational 
self-maps of $\mathbb{P}^3_\mathbb{C}$) that preserve the contact structure.

\bigskip
\noindent\emph{$2010$ Mathematics Subject Classification. --- $14E05$, $14E07$.}
\end{abstract}

\maketitle

\smallskip

\section{Introduction}\label{Sec:intro}

In this article we work on the group of birational maps that preserve contact
structures on $\mathbb{P}^3_\mathbb{C}$. 
On $\mathbb{P}^3_\mathbb{C}$ there is, up to automorphisms, only one (non-singular) 
contact structure given in homogeneous coordinates by the $1$-form $\widetilde{\vartheta}
=z_0\mathrm{d}z_1-z_1\mathrm{d}z_0+z_2\mathrm{d}z_3-z_3\mathrm{d}z_2$.
In $\mathbb{C}^3$ there is the Darboux 
$1$-form $\omega=z_0\,\mathrm{d}z_1+\mathrm{d}z_2$ that is the standard local model of 
contact forms; it thus defines a holomorphic contact structure on $\mathbb{C}^3$ 
that extends to $\mathbb{P}^3_\mathbb{C}$ meromorphically.
Note that $\omega$ has poles of order~$3$ along the hyperplane $z_3=0$. We denote by 
$\mathrm{c}(\omega)$ the (meromorphic) contact structure induced on $\mathbb{P}^3_\mathbb{C}$ 
by~$\omega$. Let us remark that actually $\omega$ is birationally conjugate to 
$\widetilde{\vartheta}_{\vert z_3=1}$ (more precisely they are conjugate via a polynomial 
automorphism in the affine chart $z_3=1$). As a result the group of birational
maps that preserve these structures are conjugate; since it is more convenient
to work with $\omega$ than with $\widetilde{\vartheta}$ we will focus on $\omega$.

\smallskip

The contact geometry has a long story. The Darboux local model $z_0\mathrm{d}z_1
+\mathrm{d}z_2$ is related to the forma\-lization of $z_0=-\frac{\mathrm{d}z_2}{\mathrm{d}z_1}$. 
 For instance if $\mathcal{S}$ is a surface in $\mathbb{C}^3$ given by $F(z_0,z_1,z_2)=0$ 
then the restriction of~$\omega$ to $\mathcal{S}$ corresponds to the implicit differential 
equation $F\left(-\frac{\partial z_2}{\partial z_1},z_1,z_2\right)=0$. A birational self-map of 
$\mathbb{P}^3_\mathbb{C}$ which preserves the contact structure (\emph{i.e.}, which sends 
the $1$-form $z_0\mathrm{d}z_1+\mathrm{d}z_2$ viewed in the affine chart $z_3=1$ onto a 
multiple of $z_0\mathrm{d}z_1+\mathrm{d}z_2$ by a rational function) is said to be a 
contact map. The space $\mathbb{C}^3$ with the contact form $\omega$ 
can be seen as an affine chart of the projectivization of the cotangent bundle 
$\mathrm{T}^*\mathbb{C}^2$ (equipped with the standard Liouville contact form). 
As a consequence there is a natural extension of any birational self-map of the $(z_1, z_2)$ 
plane (\cite{Ince})
\[
\mathcal{K}\colon\mathrm{Bir}(\mathbb{P}^2_\mathbb{C})\hookrightarrow\mathrm{Bir}
(\mathbb{C}^3)_{\mathrm{c}(\omega)},\qquad (\phi_1,\phi_2)\mapsto\left(\frac{-\frac{
\partial\phi_2}{\partial z_1}+\frac{\partial\phi_2}{\partial z_2}\,z_0}{\frac{\partial
\phi_1}{\partial z_1}-\frac{\partial \phi_1}{\partial z_2}\,z_0},\phi_1(z_1,z_2),\phi_2
(z_1,z_2)\right)
\]
where $\mathrm{Bir}(\mathbb{C}^3)_{\mathrm{c}(\omega)}$ denotes the group of 
contact birational self-maps of $\mathbb{P}^3_\mathbb{C}$. The image of $\mathcal{K}$ is 
the Klein group $\mathscr{K}$.
Klein conjectured that the group of contact maps is generated by 
$\mathscr{K}$ and the Legendre involution
\[
(z_0,z_1,z_2)\mapsto (z_1,z_0,-z_2-z_0z_1).
\]
In $2008$ Gizatullin proved this "conjecture" in the case in which the contact 
transformations are polynomial automorphisms of the affine space (\cite{Gizatullin}).
The conjecture about generators of the contact group is still open in the birational
case.

\bigskip

Let $\mathrm{G}$ be a subgroup of the group 
$\mathrm{Bir}(\mathbb{P}^n_\mathbb{C})$ of birational self-maps of $\mathbb{P}^n_\mathbb{C}$, 
and let $\beta$ be a meromorphic $p$-form on $\mathbb{P}^n_\mathbb{C}$; denote by
\[
\mathrm{G}_\beta=\big\{\phi\in\mathrm{G}\,\vert\,\phi^*\beta=\beta\big\}
\]
the subgroup of elements of $\mathrm{G}$ that preserve the form $\beta$. In the same 
spirit for $1$-forms $\beta$ we set
\[
\mathrm{G}_{\mathrm{c}(\beta)}=\big\{\phi\in\mathrm{G}\,\vert\,\phi^*\beta\wedge\beta=0\big\}.
\]
We have the obvious inclusions 
$\mathrm{G}_\beta\subset\mathrm{G}_{\mathrm{c}(\beta)}\subset\mathrm{G}$.

\smallskip

We first describe the group $\mathrm{Aut}(\mathbb{C}^3)_{\mathrm{c}(\omega)}$ of polynomial 
automorphisms of $\mathbb{C}^3$ that preserve the contact structure:

\begin{thm}
If $\eta$ is the form $\mathrm{d}\omega=\mathrm{d}z_0\wedge\mathrm{d}z_1$, then 
\[
\mathrm{Aut}(\mathbb{C}^3)_\omega\simeq \mathrm{Aut}(\mathbb{C}^2)_\eta\ltimes\mathbb{C},
\qquad
\mathrm{Aut}(\mathbb{C}^3)_{\mathrm{c}(\omega)}\simeq \mathrm{Aut}(\mathbb{C}^3)_\omega\ltimes\mathbb{C}^*.
\]
\end{thm}

Hence, as Banyaga did in the context of contact diffeomorphisms of smooth real manifolds
(\cite{Banyaga1,Banyaga2,Banyaga3}), one gets that the commutator of 
$\mathrm{Aut}(\mathbb{C}^3)_\omega$ (resp. $\mathrm{Aut}(\mathbb{C}^3)_{\mathrm{c}(\omega)}$) 
is perfect. Any automorphism of $\mathrm{Aut}(\mathbb{C}^2)$ is the composition of an 
inner automorphism and an automorphism of the field $\mathbb{C}$ (\emph{see}  
\cite{Deserti:autaut}). Following this idea we describe the group 
$\mathrm{Aut}(\mathrm{Aut}(\mathbb{C}^3)_\omega)$.

Danilov and Gizatullin proved that any finite subgroup of $\mathrm{Aut}(\mathbb{C}^2)$ is 
linearizable (\cite{DanilovGizatullin}). We obtain a similar statement:

\begin{thm}
Any finite subgroup of $\mathrm{Aut}(\mathbb{C}^3)_{\mathrm{c}(\omega)}$ is linearizable via
an element of $\mathrm{Aut}(\mathbb{C}^3)_{\mathrm{c}(\omega)}$.
\end{thm}

\medskip

We also deal with $\mathrm{Bir}(\mathbb{C}^3)_{\mathrm{c}(\omega)}$.
If $\phi$ belongs to $\mathrm{Bir}(\mathbb{C}^3)_{\mathrm{c}(\omega)}$, then 
$\phi^*\omega=V(\phi)\omega$ where $V(\phi)$ is some rational function. In 
particular one gets a map $V$ from 
$\mathrm{Bir}(\mathbb{C}^3)_{\mathrm{c}(\omega)}$ to the set of rational functions 
in $z_0$, $z_1$, $z_2$ satisfying cocycle conditions: 
$V(\phi\circ\psi)=\big(V(\phi)\circ\psi\big)\cdot V(\psi)$.

The equality $\phi^*\omega=V(\phi)\omega$ can be rewritten as the 
following system of P.D.E. 
\[
(\mathcal{S})\left\{
\begin{array}{lll}
\phi_0\frac{\partial\phi_1}{\partial z_0}+\frac{\partial\phi_2}{\partial z_0}=0 &\qquad
\qquad (\star_1)\\
\\
\phi_0\frac{\partial\phi_1}{\partial z_1}+\frac{\partial\phi_2}{\partial z_1}=V(\phi)z_0 &
\qquad\qquad (\star_2)\\
\\
\phi_0\frac{\partial\phi_1}{\partial z_2}+\frac{\partial\phi_2}{\partial z_2}=V(\phi) &\qquad
\qquad (\star_3)
\end{array}
\right.
\]
The first equation $(\star_1)$ has a special family of solutions: maps for which 
both $\phi_1$ and $\phi_2$ do not depend on~$z_0$; we can then compute $\phi_0$ from 
the two other equations. Taking $(\phi_1,\phi_2)$ in $\mathrm{Bir}(\mathbb{P}^2_\mathbb{C})$
we get in this way the group $\mathscr{K}$. 

Assume now that $\phi_1$ or $\phi_2$ depends on
$z_0$ then both depend on it and $(\mathcal{S})$ implies the following equality
\[
\frac{\frac{\partial\phi_2}{\partial z_1}-z_0\frac{\partial\phi_2}{\partial z_2}}{\frac{
\partial\phi_2}{\partial z_0}}=\frac{\frac{\partial\phi_1}{\partial z_1}-z_0\frac{
\partial\phi_1}{\partial z_2}}{\frac{\partial\phi_1}{\partial z_0}}.
\]
Let us defined $\alpha$ the map from $\mathrm{Bir}(\mathbb{C}^3)_{\mathrm{c}(\omega)}$
into the set of rational functions in $z_0$, $z_1$ and $z_2$ by: 
$\alpha(\phi)=\infty$ if~$\phi$ belongs to $\mathscr{K}$ and 
\[
\alpha(\phi)=\frac{\frac{\partial\phi_2}{\partial z_1}-z_0\frac{\partial\phi_2}{\partial z_2}}{\frac{
\partial\phi_2}{\partial z_0}}=\frac{\frac{\partial\phi_1}{\partial z_1}-z_0\frac{
\partial\phi_1}{\partial z_2}}{\frac{\partial\phi_1}{\partial z_0}}
\]
otherwise.

If $\phi_1$ and $\phi_2$ are some first integrals of the rational vector field
\[
Z_{\phi}=\alpha(\phi)\frac{\partial}{\partial z_0}-\frac{\partial}{\partial z_1}
+z_0\frac{\partial}{\partial z_2}, 
\]
one gets $\phi_0$ thanks to the first equation of $(\mathcal{S})$. Such $\phi$
is not necessary birational but only rational; never\-theless one gets a lot of 
contact birational self-maps in this way. Remark that since $\mathscr{K}$ (resp.
$\mathrm{Bir}(\mathbb{C}^3)_\omega$) is a subgroup of 
$\mathrm{Bir}(\mathbb{C}^3)_{\mathrm{c}(\omega)}$ there is a natural
left translation action of $\mathscr{K}$ (resp.
$\mathrm{Bir}(\mathbb{C}^3)_\omega$) on 
$\mathrm{Bir}(\mathbb{C}^3)_{\mathrm{c}(\omega)}$. These two actions 
admit a complete invariant:

\begin{thm}
The map $\alpha$ is a complete invariant of the left translation action of 
$\mathscr{K}$ on $\mathrm{Bir}(\mathbb{C}^3)_{\mathrm{c}(\omega)}$, that is 
for any $\phi$ and $\psi$ in $\mathrm{Bir}(\mathbb{C}^3)_{\mathrm{c}(\omega)}$ 
one has $\alpha(\phi)=\alpha(\psi)$ if and only if $\psi\phi^{-1}$ belongs to 
$\mathscr{K}$.

The map $V$ is a complete invariant of the left translation action of
$\mathrm{Bir}(\mathbb{C}^3)_\omega$ of
$\mathrm{Bir}(\mathbb{C}^3)_{\mathrm{c}(\omega)}$, \emph{i.e.} for any 
$\phi$, $\psi$ in $\mathrm{Bir}(\mathbb{C}^3)_{\mathrm{c}(\omega)}$ one
has $V(\phi)=V(\psi)$ if and only if $\psi\phi^{-1}$ belongs to 
$\mathrm{Bir}(\mathbb{C}^3)_{\omega}$.
\end{thm}

We prove that $\alpha$ is not surjective: generic linear differential equations 
of second order give linear functions that are not in the image of $\alpha$. 
Painlev\'e equations give examples of polynomials of higher degree that 
do not belong to $\mathrm{im}\,\alpha$. The map $V$ is also not surjective.

Since $\omega$ has no integral surface in $\mathbb{C}^3$ a contact birational self-map 
$\phi$ either preserves the hyperplane $z_3=0$, or blowns down $z_3=0$. This 
naturally implies the following definition: 
$\phi\in\mathrm{Bir}(\mathbb{C}^3)_{\mathrm{c}(\omega)}$ is regular at infinity
if $z_3=0$ is preserved by $\phi$ and if $\phi_{\vert z_3=0}$ is birational. One shows
that 

\begin{pro}
The set of maps of $\mathrm{Bir}(\mathbb{C}^3)_{\omega}$ 
that are regular coincides with $\mathrm{Aut}(\mathbb{P}^3_\mathbb{C})_\omega$.
\end{pro}

Let 
$\varsigma\colon\mathrm{Bir}(\mathbb{C}^3)_\omega\to\mathrm{Bir}(\mathbb{C}^2)_\eta$
be the projection onto the two first components. We say that 
$\varphi\in\mathrm{Bir}(\mathbb{C}^2)_\eta$ is exact if $\varphi$ can be
lifted via $\varsigma$ to $\mathrm{Bir}(\mathbb{C}^3)_\omega$. One
establishes the following criterion: 

\begin{thm}
A map $\varphi=(\phi_0,\phi_1)\in\mathrm{Bir}(\mathbb{C}^2)_\eta$ is 
exact if and only if the closed form $\phi_0\mathrm{d}\phi_1-z_0\mathrm{d}z_1$ has 
trivial residues. In that case 
$\phi_0\mathrm{d}\phi_1-z_0\mathrm{d}z_1=-\mathrm{d}b$ with 
$b\in\mathbb{C}(z_0,z_1)$ and 
$\phi=\big(\varphi,z_2+b(z_0,z_1)\big)\in\mathrm{Bir}(\mathbb{C}^3)_\omega$.
\end{thm}

We give a lot of examples, and even subgroups, of exact maps but also prove that
the map $\varsigma$ is not surjective:

\begin{thm}
A generic quadratic element of $\mathrm{Bir}(\mathbb{C}^2)_\eta$
is not exact.
\end{thm}

\medskip

Furthermore we look at invariant curves and surfaces. Thanks to a local argument of 
contact geometry one gets that if $\phi$ belongs to 
$\mathrm{Bir}(\mathbb{C}^3)_\omega$, if $m$ is a periodic point of $\phi$, 
and if there exists a germ of irreducible curve $\mathcal{C}$ invariant by $\phi$ 
and passing through $m$, then either $\mathcal{C}$ is a curve of periodic points, 
or~$\mathcal{C}$ is a legendrian curve. We also give a precise description of 
elements of $\mathrm{Aut}(\mathbb{C}^3)_\omega$ (resp.  
$\mathrm{Bir}(\mathbb{C}^3)_\omega$) that preserve a surface.

\medskip

Besides we deal with some group properties. Danilov proved that 
$\mathrm{Aut}(\mathbb{C}^2)_\eta$ is not simple (\cite{Danilov}); Cantat and Lamy showed 
that $\mathrm{Bir}(\mathbb{P}^2_\mathbb{C})$ is not simple (\cite{CantatLamy}). In the 
same spirit we establish that 

\begin{thm}
The groups $\mathrm{Aut}(\mathbb{C}^3)_\omega$,
 $\mathrm{Bir}(\mathbb{C}^3)_\omega$, 
$\mathrm{Aut}(\mathbb{C}^3)_{\mathrm{c}(\omega)}$, the derived group of 
$\mathrm{Aut}(\mathbb{C}^3)_\omega$ and the derived group of 
$\mathrm{Aut}(\mathbb{C}^3)_{\mathrm{c}(\omega)}$ are not simple.
\end{thm}
 
Lamy proved that $\mathrm{Aut}(\mathbb{C}^2)$ satisfies the Tits alternative
(\cite{Lamy}), then Cantat showed that $\mathrm{Bir}(\mathbb{P}^2_\mathbb{C})$ also
(\cite{Cantat:tits}). In our context one gets that

\begin{thm}
The groups $\mathrm{Aut}(\mathbb{C}^3)_\omega$, 
$\mathrm{Aut}(\mathbb{C}^3)_{\mathrm{c}(\omega)}$ and 
$\mathrm{Bir}(\mathbb{C}^3)_\omega$ satisfy the Tits alternative.
\end{thm}

\smallskip

\subsection*{Acknowledgments} We would like to thank Guy Casale for discussions 
about the non-integrability.

\section{Contact polynomial automorphisms}

\medskip

A \textbf{\textit{polynomial automorphism}} $\phi$ of $\mathbb{C}^n$ is a polynomial map of 
the type
\[
\phi\colon\mathbb{C}^n\to\mathbb{C}^n,\quad\quad\big(z_0,z_1,\ldots,z_{n-1})\mapsto(\phi_0(z_0,
z_1,\ldots,z_{n-1}),\phi_1(z_0,z_1,\ldots,z_{n-1}),\ldots,\phi_{n-1}(z_0,z_1,\ldots,z_{n-1})\big)
\]
that is bijective. The set of polynomial automorphisms of $\mathbb{C}^n$ form a group denoted 
$\mathrm{Aut}(\mathbb{C}^n)$.

The automorphisms of $\mathbb{C}^n$ of the form $(\phi_0,\phi_1,\ldots,\phi_{n-1})$ where 
$\phi_i$ depends only on $z_i$, $z_{i+1}$, $\ldots$, $z_{n-1}$ form the 
\textbf{\textit{Jonqui\`eres subgroup}} $\mathrm{J}_n\subset\mathrm{Aut}(\mathbb{C}^n)$. 
Moreover one has the inclusions
\[
\mathrm{GL}(\mathbb{C}^n)\subset\mathrm{Aff}_n\subset\mathrm{Aut}(\mathbb{C}^n)
\]
where $\mathrm{Aff}_n$ denotes the \textbf{\textit{group of affine maps}}
\[
\phi\colon(z_0,z_1,\ldots,z_{n-1})\mapsto\big(\phi_0(z_0,z_1,\ldots,z_{n-1}),\phi_1(z_0,z_1,
\ldots,z_{n-1}),\ldots,\phi_{n-1}(z_0,z_1,\ldots,z_{n-1})\big)
\]
with $\phi_i$ affine; $\mathrm{Aff}_n$ is the semi-direct product of $\mathrm{GL}(\mathbb{C}^n)$ 
with the commutative subgroups of translations. The subgroup 
$\mathrm{Tame}_n\subset\mathrm{Aut}(\mathbb{C}^n)$ generated by $\mathrm{J}_n$ and $
\mathrm{Aff}_n$ is called the \textbf{\textit{group of tame automorphisms}}.

\smallskip

\textbf{\textit{Convention:}} In all the article we denote $\mathbb{P}^n_\mathbb{C}$ by
$\mathbb{P}^n$, and we write "birational maps of $\mathbb{P}^n$" instead of 
"birational self-maps of $\mathbb{P}^n$".

\bigskip

\subsection{Contact forms and contact structures}

We recall in the context of $3$-manifolds the formalism of contact structure. Let $M$ be a 
complex $3$-manifold; we denote by $\Omega^i(M)$ the space of holomorphic $i$-forms on $M$. 
A \textbf{\textit{contact form}} on $M$ is an element $\Theta\in\Omega^1(M)$ such that the 
$3$-form $\Theta\wedge\mathrm{d}\Theta\in\Omega^3(M)$ has no zero: 
$\Theta\wedge\mathrm{d}\Theta(m)\not=0$ for any $m\in M$. For such a contact form there is 
a local model given by Darboux theorem: at each point $m$ there is a local 
biholomorphism $F\colon M,_m\to\mathbb{C}^3,_0$ such that 
$\Theta=F^*(z_0\mathrm{d}z_1+\mathrm{d}z_2)$. The $1$-form $z_0\mathrm{d}z_1+\mathrm{d}z_2$ 
is called the \textbf{\textit{standard contact form}} on $\mathbb{C}^3$; we denote it by 
$\omega$.

A \textbf{\textit{contact structure}} on the $3$-manifold $M$ is given by the following 
data:
\begin{itemize}
\smallskip
\item[$\mathfrak{i.}$] an open covering $M=\sqcup_k\mathcal{U}_k$,
\smallskip
\item[$\mathfrak{ii.}$] on each $\mathcal{U}_k$ a contact form 
$\Theta_k\in\Omega^1(\mathcal{U}_k)$,
\smallskip
\item[$\mathfrak{iii.}$] on each non-trivial intersection $\mathcal{U}_k\cap\mathcal{U}_\ell$ 
a holomorphic unit $g_{k\ell}\in\mathcal{O}^*(\mathcal{U}_k\cap\mathcal{U}_\ell)$ such that 
$\Theta_k=g_{k\ell}\Theta_\ell$.
\smallskip
\end{itemize}

A contact structure defines a holomorphic hyperplanes field 
$t\colon M\to \mathbb{P}(\mathrm{T}M)^\vee$ given for all $m\in\mathcal{U}_k$ by
\[
t(m)=\ker\Theta_k(m).
\]
As we recalled in \S \ref{Sec:intro} the compact K\"ahler manifolds having a 
contact structure are classified by Frantzen and Peternell theorem 
(\cite{FrantzenPeternell}). On $\mathbb{P}^3$ there is no contact form because 
there is no non-trivial global form. Nevertheless there are contact structures; one of 
them is given in homogeneous coordinates by the $1$-form
\[
\widetilde{\vartheta}=z_0\mathrm{d}z_1-z_1\mathrm{d}z_0+z_2\mathrm{d}z_3-z_3\mathrm{d}z_2.
\]
In that case we can take the standard covering by affine charts $\mathcal{U}_k=\{z_k=1\}$ 
and $\vartheta_k=\widetilde{\vartheta}_{\vert\mathcal{U}_k}$.

\begin{pro}\label{Pro:uniquecontactform}
Up to automorphisms of $\mathbb{P}^3$ there is only one contact structure 
on $\mathbb{P}^3$.
\end{pro}

\begin{proof}
Remark that to a contact structure on $\mathbb{P}^3$ is associated a homogeneous 
$1$-form $\beta$ on $\mathbb{C}^4$ such that $\mathcal{U}_k=\{z_k=1\}$ and 
$\Theta_k=\beta_{\vert\mathcal{U}_k}$ satisfies properties $\mathfrak{i.}$, $\mathfrak{ii.}$, 
$\mathfrak{iii.}$

\smallskip

Let $\beta$ be a contact structure on $\mathbb{P}^3$, and let 
$R=\displaystyle\sum_iz_i\frac{\partial}{\partial z_i}$ be the radial vector field. Since 
$i_R\beta=0$, to give~$\beta$ is equivalent to give $\mathrm{d}\beta$. According to 
\cite[Chapter 2, Proposition 2.1]{Jouanolou} one has $\deg \mathrm{d}\beta=0$; to give 
$\mathrm{d}\beta$ is thus equivalent to give an antisymmetric matrix of maximal rank. 
But up to conjugacy there is only one $4\times 4$ antisymmetric matrix of maximal rank.
\end{proof}

\begin{rem}
The group of linear automorphisms of $\mathbb{C}^4$ that preserve $\widetilde{\vartheta}$ 
coincides with the group of automorphisms of $\mathbb{P}^3$ that preserve 
$\mathrm{d}\widetilde{\vartheta}$; as a consequence the subgroup of 
$\mathrm{Aut}(\mathbb{P}^3)$ that preserves the contact structure associated to 
$\mathrm{d}\widetilde{\vartheta}$ is the projectivization of the symplectic group
$\mathrm{Sp}(4;\mathbb{C})$.
\end{rem}

Remark that the data of a global meromorphic $1$-form $\Theta$ on $M$ such that 
$\Theta\wedge\mathrm{d}\Theta\not\equiv 0$ induces a contact form (and a contact 
structure) on the complement of the poles and zeros of $\Theta$ and 
$\Theta\wedge\mathrm{d}\Theta$. In that case we say that $\Theta$ induces a 
\textbf{\textit{meromorphic contact structure}} on $M$.

For instance the Darboux form $\omega=z_0\mathrm{d}z_1+\mathrm{d}z_2$ induces a 
meromorphic contact structure on $\mathbb{P}^3$. In fact the forms $\omega$ and 
$\widetilde{\vartheta}_{\vert z_3=1}$ are conjugate on $\mathbb{C}^3$ via 
$\left(\frac{z_0}{2},z_1,-z_2+\frac{z_0z_1}{2}\right)$. The corresponding (meromorphic) 
contact structure are birationally conjugate on $\mathbb{P}^3$.

\medskip

\subsection{Description of contact automorphisms}

\smallskip

Let us describe $\mathrm{Aut}(\mathbb{C}^3)_\omega$. Set 
$\eta=\mathrm{d}\omega=\mathrm{d}z_0\wedge\mathrm{d}z_1$. 
Remark that the invariance of $\omega$ implies the invariance of $\eta$ and as a 
consequence the equality $(\phi_0,\phi_1)^*\eta=\eta$. 

\begin{pro}\label{Lem:chpinv}
If $\phi$ belongs to $\mathrm{Aut}(\mathbb{C}^3)_\omega$, then 
$\phi_*\frac{\partial}{\partial z_2}=\frac{\partial}{\partial z_2}$.

\smallskip

\noindent In particular if $\phi$ belongs to 
$\mathrm{Aut}(\mathbb{C}^3)_\omega$, then
\[
\phi=\big(\phi_0(z_0,z_1),\phi_1(z_0,z_1),z_2+b(z_0,z_1)\big)
\]
and the map
\[
\varsigma\colon \mathrm{Aut}(\mathbb{C}^3)_\omega\longrightarrow\mathrm{Aut}
(\mathbb{C}^2)_\eta,\qquad \big(\phi_0(z_0,z_1),\phi_1(z_0,z_1),z_2+b(z_0,z_1)
\big)\mapsto\big(\phi_0(z_0,z_1),\phi_1(z_0,z_1)\big)
\]
is a morphism.
\end{pro}

\begin{proof}
As we already mentioned, for a contact form there exists a unique vector field $\chi$, 
called Reeb vector field, such that $\omega(\chi)=~1$ and 
$i_\chi\mathrm{d}\omega=0$; here $\chi=\frac{\partial}{\partial z_2}$. If $\phi$ 
belongs to $\mathrm{Aut}(\mathbb{C}^3)_\omega$, then $\phi_*\chi=\chi$. As 
a result~$\phi$ has the following form
\[
\phi=\big(\phi_0(z_0,z_1),\phi_1(z_0,z_1),z_2+b(z_0,z_1)\big)
\]
with $(\phi_0,\phi_1)$ in $\mathrm{Aut}(\mathbb{C}^2)$ and $b$ in 
$\mathbb{C}[z_0,z_1]$.
\end{proof}

\begin{rem}
Any element of $\mathrm{Aut}(\mathbb{C}^3)_{\mathrm{c}(\omega)}$ can be written
\[
\big(\varphi_0,\varphi_1,\det\mathrm{jac}\,\varphi\, z_2 +b(z_0,z_1) \big)
\]
where $\varphi=(\varphi_0,\varphi_1)\in\mathrm{Aut}(\mathbb{C}^2)$ and
$\mathrm{d}b=(\det\mathrm{jac}\,\varphi)z_0\mathrm{d}z_1-\varphi_0\mathrm{d}\varphi_1.$ 
Let us still denote by $\varsigma$ the natural projection
\[
\varsigma\colon \mathrm{Aut}(\mathbb{C}^3)_{\mathrm{c}(\omega)}\to 
\mathrm{Aut}(\mathbb{C}^2).
\]

An element $\phi$ of $\mathrm{Bir}(\mathbb{C}^2)_\eta$ is 
\textbf{\textit{exact}} if it can be lifted via $\varsigma$ to 
$\mathrm{Bir}(\mathbb{C}^3)_\omega$, or equivalently if it belongs to 
$\mathrm{im}\,\varsigma$.

Contrary to the birational case (Theorem \ref{Thm:critere}) any element of 
$\mathrm{Aut}(\mathbb{C}^2)$ can be lifted via $\varsigma$ to 
$\mathrm{Aut}(\mathbb{C}^3)_{\mathrm{c}(\omega)}$. Since $b$ is defined up to a 
constant we do not speak about the 
$\varsigma$-lift but a $\varsigma$-lift.
\end{rem}

The following obvious statement describes the group $\mathrm{Aut}(\mathbb{C}^3)_\omega$:

\begin{pro}\label{Lem:descrrho}
Let us consider the morphism
\[
\varsigma\colon \mathrm{Aut}(\mathbb{C}^3)_\omega\longrightarrow\mathrm{Aut}(
\mathbb{C}^2)_\eta, \qquad \big(\phi_0(z_0,z_1),\phi_1(z_0,z_1),z_2+b(z_0,z_1)\big)\mapsto 
\big(\phi_0(z_0,z_1),\phi_1(z_0,z_1)\big).
\]
One has the following exact sequence
\begin{equation}\label{eq:exactsequence}
0\longrightarrow\mathbb{C}\longrightarrow\mathrm{Aut}(\mathbb{C}^3)_\omega
\stackrel{\varsigma}{\longrightarrow}\mathrm{Aut}(\mathbb{C}^2)_\eta\longrightarrow 1;
\end{equation}
more precisely 
$\ker\varsigma=\big\{(z_0,z_1,z_2+\beta)\,\vert\, \beta\in\mathbb{C}\big\}$. 
In particular
\[
\mathrm{Aut}(\mathbb{C}^3)_\omega\simeq\mathrm{Aut}(\mathbb{C}^2)_\eta\ltimes\mathbb{C}.
\]
\end{pro}

\begin{proof}
The $1$-form $\phi_0\mathrm{d}\phi_1-z_0\mathrm{d}z_1$ is a closed and polynomial one, 
so it is exact. Therefore $\varsigma$ is surjective.
\end{proof}

Let $\mathrm{G}$ be a group. The \textbf{\textit{derived group}} of $\mathrm{G}$ is 
the subgroup of $\mathrm{G}$ gene\-rated by all the commutators of $\mathrm{G}$:
\[
[\mathrm{G},\mathrm{G}]=\langle ghg^{-1}h^{-1}\,\vert\, g,\, h\in\mathrm{G}\rangle
\]
The group $\mathrm{G}$ is said to be \textbf{\textit{perfect}} if it coincides with its 
derived group, or equivalently, if the group has no nontrivial abelian quotients.

Such a property was established in the context of real smooth manifolds: Banyaga 
proved that the derived group of the group of contact diffeomorphisms is a perfect one 
(\cite{Banyaga1,Banyaga2,Banyaga3}).

\begin{thm}\label{thm:perf11}
The group $[\mathrm{Aut}(\mathbb{C}^3)_\omega,\mathrm{Aut}(\mathbb{C}^3)_\omega]$ is 
perfect. 
\end{thm}

\begin{proof}
Since $\varsigma$ is surjective (Proposition \ref{Lem:descrrho}) and 
$\mathrm{Aut}(\mathbb{C}^2)_\eta$ is perfect (\cite[Proposition 10]{FurterLamy}) the restriction 
of $\varsigma$
\[
\widetilde{\varsigma}=\varsigma_{\vert[\mathrm{Aut}(\mathbb{C}^3)_\omega,\mathrm{Aut}(\mathbb{C}^3)_\omega]}\colon
[\mathrm{Aut}(\mathbb{C}^3)_\omega,\mathrm{Aut}(\mathbb{C}^3)_\omega]\longrightarrow \mathrm{Aut}
(\mathbb{C}^2)_\eta
\]
is surjective. Let $\phi$ be in $\ker\widetilde{\varsigma}$; on the one hand 
$\phi=(z_0,z_1,z_2+\beta)$ for some $\beta$ (Proposition \ref{Lem:descrrho}), and on the 
other hand $\phi$ is a product of commutators hence $\beta=0$. We thus have the following 
exact sequence
\[
0\longrightarrow[\mathrm{Aut}(\mathbb{C}^3)_\omega,\mathrm{Aut}(\mathbb{C}^3)_\omega]
\longrightarrow\mathrm{Aut}(\mathbb{C}^2)_\eta\longrightarrow 1
\]
and 
$[\mathrm{Aut}(\mathbb{C}^3)_\omega,\mathrm{Aut}(\mathbb{C}^3)_\omega] \simeq\mathrm{Aut}(\mathbb{C}^2)_\eta$ 
which is perfect (\cite[Proposition 10]{FurterLamy}).
\end{proof}

\smallskip

We will now describe $\mathrm{Aut}(\mathbb{C}^3)_{\mathrm{c}(\omega)}$.
Let us recall that $\mathrm{Aut}(\mathbb{C}^2)$ is generated by $\mathrm{J}_2$ and 
$\mathrm{Aff}_2$ (\emph{see} \cite{Jung}). This implies that $\mathrm{Aff}_2$ and 
\[
[\mathrm{J}_2,\mathrm{J}_2]=\big\{(z_0+\beta,z_1+P (z_0))\,\vert\,\beta\in\mathbb{C},\,P\in\mathbb{C}[z_0]\big\}.
\]
generate $\mathrm{Aut}(\mathbb{C}^2)$.

\begin{pro}\label{pro:genautC3}
The group $\mathrm{Aut}(\mathbb{C}^3)_{\mathrm{c}(\omega)}$ is generated by $\mathcal{A}$ 
and $\mathcal{E}$ where
\[
\mathcal{E}=\big\{\varsigma\text{-lifts of }\mathfrak{e}\,\vert \,\mathfrak{e}\in[\mathrm{J}_2,\mathrm{J}_2]\big\}
\qquad
\text{and}
\qquad
\mathcal{A}=\big\{ \varsigma\text{-lifts of }\mathfrak{a}\,\vert \,\mathfrak{a}\in\mathrm{Aff}_2\big\}.
\]
\end{pro}

\begin{proof}
Let $\varphi$ be a polynomial automorphism of $\mathbb{C}^2$ and let $\phi$ be a $\varsigma$-lift 
of $\varphi$ to $\mathrm{Aut}(\mathbb{C}^3)_{\mathrm{c}(\omega)}$
\[
\phi=\big(\varphi,\det\mathrm{jac}\,\varphi z_2+b(z_0,z_1)\big)
\]
with $b$ in $\mathbb{C}[z_0,z_1]$. One can write $\varphi$ as 
$\mathfrak{a}_1\mathfrak{e}_1\mathfrak{a}_2\mathfrak{e}_2\ldots \mathfrak{a}_s\mathfrak{e}_s$ 
where $\mathfrak{a}_i$ belongs to $\mathrm{Aff}_2$ and $\mathfrak{e}_i$ to $[\mathrm{J}_2,\mathrm{J}_2]$. 
Let us now consider $A_i$ a $\varsigma$-lift of $\mathfrak{a}_i$,  $E_i=(\mathfrak{e}_i,z_2+d_i)$ a 
$\varsigma$-lift of $\mathfrak{e}_i$. Then $A_1E_1A_2E_2\ldots A_sE_s$ belongs to 
$\mathrm{Aut}(\mathbb{C}^3)_{\mathrm{c}(\omega)}$, and up to composition by an element 
$(z_0,z_1,z_2+\beta)\in\mathcal{A}$ one has
\[
\phi=A_1E_1A_2E_2\ldots  A_sE_s.
\]
\end{proof}

\begin{pro}\label{pro:desccont}
One has
\[
\mathrm{Aut}(\mathbb{C}^3)_{\mathrm{c}(\omega)}\simeq\mathrm{Aut}(\mathbb{C}^3)_\omega\ltimes\mathbb{C}^*.
\]
\end{pro}

\begin{proof}
Let us consider an element $\phi$ of $\mathrm{Aut}(\mathbb{C}^3)_{\mathrm{c}(\omega)}$, then 
$\phi^*\omega=V(\phi)\omega$ for some polynomial $V(\phi)$. As $\omega$ and~$\phi^*\omega$ do 
not vanish, $V(\phi)$ does not vanish; therefore $V(\phi)=\lambda\in\mathbb{C}^*$. Let us write 
$\phi$ as follows:
\[
\phi=(\lambda z_0,z_1,\lambda z_2)\circ\widetilde{\phi};
\]
of course $\widetilde{\phi}^*\omega=\omega$.
\end{proof}

\begin{thm}\label{thm:perf13}
The derived group 
$[\mathrm{Aut}(\mathbb{C}^3)_{\mathrm{c}(\omega)},\mathrm{Aut}(\mathbb{C}^3)_{\mathrm{c}(\omega)}]$ 
of $\mathrm{Aut}(\mathbb{C}^3)_{\mathrm{c}(\omega)}$ is perfect.
\end{thm}

\begin{proof}
According to Proposition \ref{pro:desccont} an element $\phi$ of 
$\mathrm{Aut}(\mathbb{C}^3)_{\mathrm{c}(\omega)}$ can be written
\[
\big(\lambda\phi_0,\phi_1,\lambda z_2+\lambda b\big)
\]
with $\lambda\in\mathbb{C}^*$ and $(\phi_0,\phi_1,z_2+b)\in\mathrm{Aut}(\mathbb{C}^3)_\omega$. 
Denote by $\varphi$ the element of $\mathrm{Aut}(\mathbb{C}^2)$ given by $(\phi_0,\phi_1)$.
If $\phi$ belongs to $\ker\varsigma$, then $\lambda=1$, 
$\varphi=\mathrm{id}$ and $b\in\mathbb{C}$, that is $\ker\varsigma\simeq\mathbb{C}$ and
\begin{equation}\label{eq:2dsuiteex}
\mathbb{C}\longrightarrow \mathrm{Aut}(\mathbb{C}^3)_{\mathrm{c}(\omega)}
\stackrel{\varsigma}{\longrightarrow}\mathrm{Aut}(\mathbb{C}^2)\longrightarrow 1.
\end{equation}
Since $\mathrm{Aut}(\mathbb{C}^2)_\eta$ is perfect the restriction of $\varsigma$ to 
$[\mathrm{Aut}(\mathbb{C}^3)_{\mathrm{c}(\omega)},\mathrm{Aut}(\mathbb{C}^3)_{\mathrm{c}(\omega)}]$ 
induces the following exact sequence
\[
0\longrightarrow[\mathrm{Aut}(\mathbb{C}^3)_{\mathrm{c}(\omega)},\mathrm{Aut}
(\mathbb{C}^3)_{\mathrm{c}(\omega)}]\longrightarrow\mathrm{Aut}(\mathbb{C}^2)_\eta\longrightarrow 1
\]
and $[\mathrm{Aut}(\mathbb{C}^3)_{\mathrm{c}(\omega)},\mathrm{Aut}(\mathbb{C}^3)_{\mathrm{c}(\omega)}]
\simeq \mathrm{Aut}(\mathbb{C}^2)_\eta$. One concludes as previously with 
\cite[Proposition 10]{FurterLamy}.
\end{proof}

\smallskip

Let us now deal with the finite subgroups of $\mathrm{Aut}(\mathbb{C}^3_{\mathrm{c}(\omega)}$.

\begin{pro}\label{pro:periodic}
Any element of $\mathrm{Aut}(\mathbb{C}^2)_\eta$ of period $\ell$ 
lifts via $\varsigma$ to a unique element of 
$\mathrm{Aut}(\mathbb{C}^3)_\omega$ of period~$\ell$. 
\end{pro}

\begin{proof}
Let us consider an element $\varphi=\big(\phi_0(z_0,z_1),\phi_1(z_0,z_1)\big)$ 
of $\mathrm{Aut}(\mathbb{C}^2)_\eta$. According to Proposition \ref{Lem:descrrho} 
there exists $b\in\mathbb{C}[z_0,z_1]$ such that 
$\big(\phi_0(z_0,z_1),\phi_1(z_0,z_1),z_2+b(z_0,z_1)+\mu\big)$ belongs to 
$\mathrm{Bir}(\mathbb{C}^3)_\omega$ for any $\mu\in\mathbb{C}$. 
Assume that $\varphi$ is of prime order $\ell$; let us prove that there 
exists a unique $\gamma\in\mathbb{C}$ such that
\[
\big(\phi_0,\phi_1,z_2+b(z_0,z_1)+\gamma\big)
\]
is of order $\ell$.

Assume for simplicity that $\ell=2$ (but a similar argument works for any 
$\ell$). Let us recall that the following equality holds
\begin{equation}\label{ordrefini1}
z_0\mathrm{d}z_1-\phi_0\mathrm{d}\phi_1=\mathrm{d}b
\end{equation}
Applying $\phi$ to this equality one gets
\begin{equation}\label{ordrefini2}
\phi_0\mathrm{d}\phi_1-z_0\mathrm{d}z_1=\mathrm{d}(b\circ\varphi)
\end{equation}
We add (\ref{ordrefini1}) and (\ref{ordrefini2}) and obtain that 
$b+b\circ\phi$ is a constant $\beta$. Furthermore
\[
\big(\phi_0(z_0,z_1),\phi_1(z_0,z_1),z_2+b(z_0,z_1)+\mu\big)^2=\big(z_0,z_1,
z_2+2\gamma+b+b\circ\varphi\big)=(z_0,z_1,z_2+2\gamma+\beta)
\]
so as soon as $\gamma=-\beta/2$ one has 
$\big(\phi_0(z_0,z_1),\phi_1(z_0,z_1),z_2+b(z_0,z_1)+\mu\big)^2=\mathrm{id}$.
\end{proof}

\begin{pro}\label{pro:finitesubgroupaut}
A finite subgroup of $\mathrm{Aut}(\mathbb{C}^2)$ can be lifted to a finite subgroup 
of $\mathrm{Aut}(\mathbb{C}^3)_{\mathrm{c}(\omega)}$.
\end{pro}

\begin{proof}
Let $\mathrm{H}$ be a finite subgroup of $\mathrm{Aut}(\mathbb{C}^2)$. The group $\mathrm{H}$ 
is linearizable (\cite{DanilovGizatullin}) hence has a fixed point~$p$. Since the translations 
belong to $\mathrm{Aut}(\mathbb{C}^2)$ one can assume that $p=(0,0)$. Let us consider the lifts 
of all elements of $\mathrm{H}$ in 
$\big\{\phi\in\mathrm{Aut}(\mathbb{C}^3)_{\mathrm{c}(\omega)}\,\vert\,\phi(0)=0\big\}$; they 
form a group isomorphic to $\mathrm{H}$ so is in particular finite.
\end{proof}

\begin{rem}
Any subgroup $\mathrm{G}$ of $\mathrm{Aut}(\mathbb{C}^2)$ that preserves $(0,0)$ can be lifted 
to a subgroup of  $\mathrm{Aut}(\mathbb{C}^3)_{\mathrm{c}(\omega)}$ isomorphic to $\mathrm{G}$.
\end{rem}

\begin{thm}\label{thm:finitegp}
Any finite subgroup of $\mathrm{Aut}(\mathbb{C}^3)_{\mathrm{c}(\omega)}$ is linearizable 
via an element of $\mathrm{Aut}(\mathbb{C}^3)_{\mathrm{c}(\omega)}$.
\end{thm}

\begin{proof}
Let $\mathrm{G}$ be a finite subgroup of $\mathrm{Aut}(\mathbb{C}^3)_{\mathrm{c}(\omega)}$. The 
group $\mathrm{G}$ is isomorphic to $\mathrm{H}=\varsigma(\mathrm{G})$ which is thus a 
finite subgroup of~$\mathrm{Aut}(\mathbb{C}^2)$. There exists a map $h\in\mathrm{Aut}(\mathbb{C}^2)$ 
that linearizes $\mathrm{H}$ (\emph{see} \cite{DanilovGizatullin}); as a result $\mathrm{H}$ 
has a fixed point $p$ and up to translations one can suppose that $p=(0,0)$. Note that $h(0)=0$. 
The lift of $h$ in $\big\{\phi\in\mathrm{Aut}(\mathbb{C}^3)_{\mathrm{c}(\omega)}\,\vert\,\phi(0)=0\big\}$ 
linearizes $\mathrm{G}$.
\end{proof}

\medskip

\subsection{Automorphisms group}\label{subsec:autaut}

Let us first introduce some notations. The group of the field automorphisms of $\mathbb{C}$ 
acts on $\mathrm{Aut}(\mathbb{C}^n)$ (resp. $\mathrm{Bir}(\mathbb{P}^n)$): if $f$ is 
an element of~$\mathrm{Aut}(\mathbb{C}^n)$ and if $\xi$ is a field automorphism we denote 
by~${}^{\xi}\!\, f$ the element obtained by letting~$\xi$ acting on $f$. Using the 
structure of amalgamated product of $\mathrm{Aut}(\mathbb{C}^2)$, the automorphisms 
of this group have been described (\cite{Deserti:autaut}): let $\varphi$ be an 
automorphism of~$\mathrm{Aut}(\mathbb{C}^2)$; there exist a polynomial 
automorphism~$\psi$ of $\mathbb{C}^2$ and a field automorphism $\xi$ such that
\[
\forall\,f\in\mathrm{Aut}(\mathbb{C}^2)\qquad \varphi(f)={}^{\xi}\!\,(\psi f\psi^{-1}).
\]
Even if $\mathrm{Bir}(\mathbb{P}^2)$ has not the same structure as 
$\mathrm{Aut}(\mathbb{C}^2)$ (\emph{see} Appendix of \cite{CantatLamy}) the 
automorphisms group of $\mathrm{Bir}(\mathbb{P}^2)$ can be described and a 
similar result is obtained (\cite{Deserti:autbir}). 

We now would like to describe the group $\mathrm{Aut}\big(\mathrm{Aut}(\mathbb{C}^3)_\omega\big)$.
Let us recall that the \textbf{\textit{center}} of a group $\mathrm{G}$, denoted 
$Z(\mathrm{G})$, is the set of elements that commute with every element of $\mathrm{G}$.

\begin{pro}\label{Pro:centrea}
The center of $\mathrm{Aut}(\mathbb{C}^3)_\omega$ is isomorphic to $\mathbb{C}$:
\[
Z(\mathrm{Aut}(\mathbb{C}^3)_\omega)=
\big\{(z_0,z_1,z_2+\beta)\,\vert\,\beta\in\mathbb{C}\big\}
\]
and the center of $\mathrm{Aut}(\mathbb{C}^3)_{\mathrm{c}(\omega)}$ is trivial.
\end{pro}

As $\mathrm{Aut}(\mathbb{C}^3)_\omega\simeq\mathrm{Aut}(\mathbb{C}^2)_\eta\ltimes\mathbb{C}$ 
Proposition \ref{Pro:centrea} implies the following statement:

\begin{cor}\label{cor:centrea}
The quotient of $\mathrm{Aut}(\mathbb{C}^3)_\omega$ by its center is isomorphic to 
$\mathrm{Aut}(\mathbb{C}^2)_\eta$.
\end{cor}

\begin{lem}
One has the following isomorphism 
\[
\mathrm{Hom}(\mathrm{Aut}(\mathbb{C}^3)_\omega,\mathbb{C})\simeq \mathrm{Hom}
(\mathbb{C},\mathbb{C})
\]
where $\mathrm{Hom}(\mathbb{C},\mathbb{C})$ denotes the homomorphisms of the
additive group $\mathbb{C}$.
\end{lem}

\begin{proof}
Note that if $\phi$ belongs to 
$[\mathrm{Aut}(\mathbb{C}^3)_\omega,\mathrm{Aut}(\mathbb{C}^3)_\omega]$, then the 
last component of $\phi$ is well defined (that is not defined modulo 
a constant). Besides 
$\mathrm{Aut}(\mathbb{C}^3)_\omega\simeq\mathrm{Aut}(\mathbb{C}^2)_\eta\ltimes\mathbb{C}$ 
and $\mathrm{Aut}(\mathbb{C}^2)_\eta$ is perfect thus 
\[
\quad ^{\textstyle \mathrm{Aut}(\mathbb{C}^3)_\omega}\Big/_{\textstyle 
[\mathrm{Aut}(\mathbb{C}^3)_\omega,\mathrm{Aut}(\mathbb{C}^3)_\omega]}\simeq\mathbb{C}
\]
and
\[
\xymatrix{
\mathrm{Aut}(\mathbb{C}^3)_\omega\ar[dd]\simeq\mathrm{Aut}(\mathbb{C}^2)_\eta
\ltimes\mathbb{C}\ar[ddrr] & & \\
 & & \\
\quad ^{\textstyle \mathrm{Aut}(\mathbb{C}^3)_\omega}\Big/_{\textstyle 
[\mathrm{Aut}(\mathbb{C}^3)_\omega,\mathrm{Aut}(\mathbb{C}^3)_\omega]}\ar[rr]^-{\sim} & & \mathbb{C}
}
\]
We conclude by noting that any element of 
$\mathrm{Hom}(\mathrm{Aut}(\mathbb{C}^3)_\omega,\mathbb{C})$ acts trivially on $\phi$. 
\end{proof}

\begin{rem}
An element $c$ of $\mathrm{Hom}(\mathrm{Aut}(\mathbb{C}^3)_\omega,\mathbb{C})$ 
acts on $\mathrm{Aut}(\mathbb{C}^3)_\omega$ as follows
\[
\big(\phi_0,\phi_1,z_z+b(z_0,z_1)\big)\to \big(\phi_0,\phi_1,z_2+b(z_0,z_1)+c(\phi)\big)
\]
\end{rem}

\begin{defi}
Let $\mathrm{H}$ be a normal subgroup of a group $\mathrm{G}$. We say that 
an automorphism of $\mathrm{H}$ of the form $\phi\mapsto \varphi\phi\varphi^{-1}$, 
with $\varphi$ in $\mathrm{G}$, is \textbf{\textit{$\mathrm{G}$-inner}}.
\end{defi}

\begin{thm}\label{thm:autautC3omega}
The group $\mathrm{Aut}\big(\mathrm{Aut}(\mathbb{C}^3)_\omega\big)$ is 
generated by the automorphisms group of the field~$\mathbb{C}$, the group 
of $\mathrm{Aut}(\mathbb{C}^3)_{\mathrm{c}(\omega)}$-inner automorphisms and the
action of $\mathrm{Hom}(\mathbb{C},\mathbb{C})$.
\end{thm}

\begin{proof}
Consider an element $\psi$ of 
$\mathrm{Aut}\big(\mathrm{Aut}(\mathbb{C}^3)_\omega\big)$. For any 
$\phi=(\varphi_\phi,z_2+T_\phi(z_0,z_1))$ one has  
\[
\psi(\phi)=\big(\widetilde{\varphi_\phi},z_2+\Delta_\phi(z_0,z_1)\big). 
\]
In particular $\psi$ induces an automorphism~$\psi_0$ of 
$\mathrm{Aut}(\mathbb{C}^2)_\eta$; indeed since $\psi$ is an automorphism of 
$\mathrm{Aut}(\mathbb{C}^3)_\omega$, it preserves 
$Z(\mathrm{Aut}(\mathbb{C}^3)_\omega)$ and so, from Corollary \ref{cor:centrea} 
induces an automorphism of $\mathrm{Aut}(\mathbb{C}^2)_\eta$.

According to Theorem \ref{thm:autauteta} one can assume that 
$\psi_0=\mathrm{id}$ up to the action of an automorphism of the 
field~$\mathbb{C}$ and up to conjugacy by an $\mathrm{Aut}(\mathbb{C}^2)$-inner 
automorphism, {\it i.e.}
\[
\psi(\phi)=\big(\varphi_\phi,z_2+\Delta_\phi(z_0,z_1)\big)
\]
Set $\phi^{-1}=\big(\varphi_\phi^{-1},z_2+T_{\phi^{-1}}(z_0,z_1)\big)$. On the 
one hand
$\phi^{-1}\circ\phi=\big(\mathrm{id},z_2+T_\phi(z_0,z_1)+T_{\phi^{-1}}(\varphi_\phi)\big)$
so
\begin{equation}\label{eq:ast}
T_\phi+T_{\phi^{-1}}(\varphi_\phi)=0
\end{equation}
and on the other hand
\[
\psi(\phi\circ\phi^{-1})=\big(\mathrm{id},z_2+T_{\phi^{-1}}(z_0,z_1)+
\Delta_\phi\,\varphi_\phi^{-1}\big)
\]
belongs to 
$\mathrm{Aut}(\mathbb{C}^3)_\omega$ hence $T_{\phi^{-1}}+\Delta_\phi\varphi_\phi^{-1}$ 
is a constant. This, combined with $(\ref{eq:ast})$, implies that 
$\Delta_\phi=T_\phi+c_\phi$, where $c_\phi$ is a constant, and yields to a morphism 
from $\mathrm{Aut}(\mathbb{C}^3)_\omega$ to $\mathbb{C}$:
\[
\mathrm{Aut}(\mathbb{C}^3)_\omega\to\mathbb{C},\qquad \phi\mapsto c_\phi.
\]

\smallskip

Consider an homomorphism
\[
\rho\colon\mathrm{Aut}(\mathbb{C}^3)_\omega\to \mathbb{C},\qquad \phi\mapsto\rho_\phi.
\]
Let us define 
$\psi\colon\mathrm{Aut}(\mathbb{C}^3)_\omega\to\mathrm{Aut}(\mathbb{C}^3)_\omega$ 
by: 
\[
\psi(\phi)=\big(\phi_0(z_0,z_1),\phi_1(z_0,z_1),z_2+b(z_0,z_1)+\rho_\phi\big)
\]
where 
$\phi=\big(\phi_0(z_0,z_1),\phi_1(z_0,z_1),z_2+b(z_0,z_1)\big)\in\mathrm{Aut}(\mathbb{C}^3)_\omega$.
One can check that $\psi$ belongs to $\mathrm{Aut}(\mathrm{Aut}(\mathbb{C}^3)_\omega)$.
\end{proof}

\section{Contact birational maps}

\medskip

A \textbf{\textit{rational map}} of $\mathbb{P}^n$ can be written
\[
\phi\colon\mathbb{P}^n\dashrightarrow\mathbb{P}^n\quad\quad \big(z_0:z_1:
\ldots:z_n)\dashrightarrow(\phi_0(z_0,z_1,\ldots,z_n):\phi_1(z_0,z_1,\ldots,z_n):\ldots:
\phi_n(z_0,z_1,\ldots,z_n)\big)
\]
where the $\phi_i$'s are homogeneous polynomials of the same degree $\geq 1$ and without 
common factor of positive degree. The \textbf{\textit{degree}} of $\phi$ is by definition 
the degree of the $\phi_i$. A \textbf{\textit{birational map}} of $\mathbb{P}^n$ 
is a rational map that admits a rational inverse. Of course 
$\mathrm{Aut}(\mathbb{C}^n)$ is a subgroup of $\mathrm{Bir}(\mathbb{P}^n)$. An 
other natural subgroup of $\mathrm{Bir}(\mathbb{P}^n)$ is the group 
$\mathrm{Aut}(\mathbb{P}^n)\simeq\mathrm{PGL}(n+1;\mathbb{C})$ of automorphisms 
of $\mathbb{P}^n$.

The \textbf{\textit{indeterminacy set}} $\mathrm{Ind}\,\phi$ of $\phi$ is the set of the 
common zeros of the $\phi_i$'s. The \textbf{\textit{exceptional set}} $\mathrm{Exc}\,\phi$ 
of~$\phi$ is the (finite) union of subvarieties $M_i$ of $\mathbb{P}^n$ such that 
$\phi$ is not injective on any open subset of $M_i$.

\smallskip

Let us extend the definition of Jonqui\`eres group we gave in the case of polynomial 
automorphisms of $\mathbb{C}^n$ to the case of birational 
maps of~$\mathbb{P}^2$: the \textbf{\textit{Jonqui\`eres group}}, denoted~$\mathcal{J}$, 
is the group of birational maps of $\mathbb{P}^2$ that preserve a 
pencil of rational curves. Since two pencils of rational curves are birationally conjugate, 
$\mathcal{J}$ does not depend, up to conjugacy, of the choice of the pencil. In other words 
one can decide, up to birational conjugacy, that $\mathcal{J}$ is in the affine chart $z_2=1$ 
the maximal group of birational maps that preserve the fibration $z_1=$ cst. An element 
$\varphi$ of $\mathcal{J}$ permutes the fibers of the fibration thus induces an automorphism 
of the base $\mathbb{P}^1$; note that if the fibration is fiberwise invariant, 
$\varphi$ acts as an homography in the generic fibers. Hence $\mathcal{J}$ can be identified 
with the semi-direct product $\mathrm{PGL}(2;\mathbb{C}(z_1))\rtimes\mathrm{PGL}(2;\mathbb{C})$.

\smallskip

We study the birational maps $\phi=(\phi_0,\phi_1,\phi_2)$ defined on 
$\mathbb{C}^3=(z_3=1)\subset\mathbb{P}^3$ that preserve either the contact 
standard form $\omega$, or the contact structure $\mathrm{c}(\omega)$ associated to 
$\omega$. In other words we would like to describe the groups 
$\mathrm{Bir}(\mathbb{C}^3)_\omega$ and 
$\mathrm{Bir}(\mathbb{C}^3)_{\mathrm{c}(\omega)}$ and also their elements.

Let us now illustrate a fundamental difference between 
$\mathrm{Bir}(\mathbb{C}^3)_\omega$ and 
$\mathrm{Bir}(\mathbb{C}^3)_{\mathrm{c}(\omega)}$: the first group preserves 
the fibration associated to $\frac{\partial}{\partial z_2}$ whereas the second 
doesn't.

\begin{pro}\label{Lem:chpinv2}
If $\phi$ belongs to $\mathrm{Bir}(\mathbb{C}^3)_\omega$, then 
$\phi_*\frac{\partial}{\partial z_2}=\frac{\partial}{\partial z_2}$.

\smallskip

\noindent In particular if $\phi$ belongs to 
$\mathrm{Bir}(\mathbb{C}^3)_\omega$, then
\[
\phi=\big(\phi_0(z_0,z_1),\phi_1(z_0,z_1),z_2+b(z_0,z_1)\big)
\]
and the map
\[
\varsigma\colon \mathrm{Bir}(\mathbb{C}^3)_\omega\longrightarrow\mathrm{Bir}
(\mathbb{C}^2)_\eta,\qquad \big(\phi_0(z_0,z_1),\phi_1(z_0,z_1),z_2+b(z_0,z_1)
\big)\mapsto\big(\phi_0(z_0,z_1),\phi_1(z_0,z_1)\big)
\]
is a morphism.
\end{pro}

\begin{rem}
The proof is similar to the proof of Proposition \ref{Lem:chpinv}.
\end{rem}

\begin{rem}\label{egsansfib}
The first assertion of Proposition \ref{Lem:chpinv2} is not true for the group 
$\mathrm{Bir}(\mathbb{C}^3)_{\mathrm{c}(\omega)}$; indeed let us consider the 
map $\psi$ defined by
\[
\psi=\left(\frac{z_0}{(1+z_2)^2},z_1,\frac{z_2}{1+z_2}\right);
\]
it belongs to $\mathrm{Bir}(\mathbb{C}^3)_{\mathrm{c}(\omega)}$ and does not 
preserve the fibration associated to the vector field $\frac{\partial}{\partial z_2}$.
\end{rem}

\medskip

\subsection{A P.D.E. approach}

Let $\phi=(\phi_0,\phi_1,\phi_2)$ be in $\mathrm{Bir}(\mathbb{C}^3)_{\mathrm{c}(\omega)}$; 
then $\phi^*\omega=V(\phi)\omega$ for some rational function $V(\phi)$. One
inherits a map $V$ from $\mathrm{Bir}(\mathbb{C}^3)_{\mathrm{c}(\omega)}$ into the set of 
rational functions in $z_0$, $z_1$ and $z_2$. The equality $\phi^*\omega=V(\phi)\omega$ 
gives the following system $(\star)$ of P. D. E.:
\[
\left\{
\begin{array}{lll}
\phi_0\frac{\partial\phi_1}{\partial z_0}+\frac{\partial\phi_2}{\partial z_0}=0 &\qquad
\qquad (\star_1)\\
\\
\phi_0\frac{\partial\phi_1}{\partial z_1}+\frac{\partial\phi_2}{\partial z_1}=V(\phi) z_0 &
\qquad\qquad (\star_2)\\
\\
\phi_0\frac{\partial\phi_1}{\partial z_2}+\frac{\partial\phi_2}{\partial z_2}=V(\phi) &\qquad
\qquad (\star_3)
\end{array}
\right.
\]
Thanks to $(\star_2)$ and $(\star_3)$ one gets
\[
\phi_0\left(\frac{\partial\phi_1}{\partial z_1}-z_0\frac{\partial\phi_1}{\partial z_2}
\right)+\left(\frac{\partial\phi_2}{\partial z_1}-z_0\frac{\partial\phi_2}{\partial z_2}
\right)=0 \qquad\qquad (\star_4)
\]

Equation $(\star_1)$ has a special family of solutions: maps for which both $\phi_1$ or 
$\phi_2$ do not depend on $z_0$ (note that if $\phi_1$ (resp. $\phi_2$) does not depend
on $z_0$ then $(\star_1)$ implies that $\phi_2$ (resp. $\phi_1$) also); in that case we 
can then compute $\phi_0$ thanks to $(\star_4)$. Taking $(\phi_1,\phi_2)$ in 
$\mathrm{Bir}(\mathbb{P}^2)$ we get elements in $\mathrm{im}\,\mathcal{K}$; 
we will called this family of solutions \textbf{\textit{Klein family}}. Note 
that this family is a group denoted $\mathscr{K}$, 
the \textbf{\textit{Klein group}}.

\begin{pro}
The elements of $\mathscr{K}$ are of the following type
\[
\left(\frac{-\frac{\partial\phi_2}{\partial z_1}+z_0\frac{\partial\phi_2}{\partial z_2}
}{\frac{\partial\phi_1}{\partial z_1}-z_0\frac{\partial\phi_1}{\partial z_2}},\phi_1(z_1,z_2),
\phi_2(z_1,z_2)\right)
\]
with $(\phi_1,\phi_2)$ in $\mathrm{Bir}(\mathbb{P}^2)$.
\end{pro}

Assume now that $\phi_1$ or $\phi_2$ really depends on $z_0$ (\emph{i.e.} that $\phi$ does
not belong to the Klein family). Then $(\star_1)$ and $(\star_4)$ imply
\[
\left(\frac{\partial\phi_2}{\partial z_1}-z_0\frac{\partial\phi_2}{\partial z_2}\right)
\frac{\partial\phi_1}{\partial z_0}=\left(\frac{\partial\phi_1}{\partial z_1}-z_0\frac{
\partial\phi_1}{\partial z_2}\right)\frac{\partial\phi_2}{\partial z_0}\qquad\qquad (
\star_5)
\]
One can rewrite $(\star_5)$ as
\[
\frac{\frac{\partial\phi_2}{\partial z_1}-z_0\frac{\partial\phi_2}{\partial z_2}}{\frac{
\partial\phi_2}{\partial z_0}}=\frac{\frac{\partial\phi_1}{\partial z_1}-z_0\frac{
\partial\phi_1}{\partial z_2}}{\frac{\partial\phi_1}{\partial z_0}}.
\]
Denote by $\alpha$ the map from $\mathrm{Bir}(\mathbb{C}^3)_{\mathrm{c}(\omega)}$ to
the set of rational functions in $z_0$, $z_1$ and $z_2$ defined by $\alpha(\phi)=\infty$
if $\phi$ belongs to $\mathscr{K}$ and 
\[
\alpha(\phi)=\frac{\frac{\partial\phi_2}{\partial z_1}-z_0\frac{\partial\phi_2}{\partial 
z_2}}{\frac{\partial\phi_2}{\partial z_0}}=\frac{\frac{\partial\phi_1}{\partial z_1}-z_0
\frac{\partial\phi_1}{\partial z_2}}{\frac{\partial\phi_1}{\partial z_0}}
\]
otherwise.

If $\phi_1$ and $\phi_2$ are some first integrals of 
\[
Z_{\phi}=\alpha(\phi)\frac{\partial}{\partial z_0}-\frac{\partial}{\partial z_1}
+z_0\frac{\partial}{\partial z_2}, 
\]
then $(\star_5)$ is satisfied. One thus gets $\phi_0$ from $(\star_1)$. Note
that such a $\phi$ is not always birational. But one can get a lot of birational 
examples in this way.

\smallskip

For instance when $\alpha(\phi)\equiv 0$ one obtains a family of 
rational maps solutions of $(\star)$ and Legendre involution is one of them. 
The set of birational maps of that family 
is called \textbf{\textit{Legendre family}}, \emph{i.e.} it is the set of 
birational maps of the following form
\[
\left(-\frac{\frac{\partial}{\partial z_0}\left(\phi_2\big(z_0,-(z_2+z_0z_1)\big)\right)}
{\frac{\partial}{\partial z_0}\left(\phi_1\big(z_0,-(z_2+z_0z_1)\big)\right)},\phi_1
\big(z_0,-(z_2+z_0z_1)\big),\phi_2\big(z_0,-(z_2+z_0z_1)\big)\right).
\]

\begin{rem}
The Legendre family composed with the Legendre involution (right 
composition) yields to the Klein family.
\end{rem}

\begin{defi}
Let $\gamma$ be an irreducible curve; $\gamma$ is a \textbf{\textit{legendrian curve}} 
if $s_\gamma^*\omega=0$ where $s_\gamma$ denotes a local parametrization of $\gamma$.
\end{defi}

\begin{rem}
Elements of the Klein family preserve the fibration 
$\big\{z_1 =\text{ cst},\,z_2 =\text{ cst}\big\}$; note that its fibers are legendrian 
curves. The Legendre involution sends the fibration 
$\big\{z_0 =\text{ cst},\,z_2+z_0z_1 =\text{ cst}\big\}$ onto 
$\big\{z_1 =\text{ cst},\,z_2 =\text{ cst}\big\}$. Then of course if one conjugates 
the Klein family by the Legendre involution one gets a family that 
preserves the fibration by legendrian curves 
$\big\{z_0 =\text{ cst},\,z_2+z_0z_1 =\text{ cst}\big\}$.
\end{rem}

A direct computation implies:

\begin{pro}\label{pro:firstip}
Let $\phi=(\phi_0,\phi_1,\phi_2)$ be a contact birational map of $\mathbb{P}^3$. 

The map $\phi$ conjugates the foliation induced by $Z_{\phi}$ to the foliation 
induced by $\frac{\partial}{\partial z_0}$.

As a consequence the field of the rational first integrals of $Z_{\phi}$ is generated
by $\phi_1$ and $\phi_2$.
\end{pro}

\smallskip

The left translation action of $\mathscr{K}$ on 
$\mathrm{Bir}(\mathbb{C}^3)_{\mathrm{c}(\omega)}$ is given by 
\[
(\psi,\phi)\in\mathscr{K}\times\mathrm{Bir}(\mathbb{C}^3)_{\mathrm{c}(\omega)}
\longrightarrow \psi\phi\in\mathrm{Bir}(\mathbb{C}^3)_{\mathrm{c}(\omega)}.
\]
Take $\phi$ and $\psi$ in $\mathrm{Bir}(\mathbb{C}^3)_{\mathrm{c}(\omega)}$
such that $\alpha(\phi)=\alpha(\psi)$, then $\psi_1$ and $\psi_2$ are first
integrals of $Z_\phi$ and by Proposition \ref{pro:firstip}
\[
\psi_1=\varphi_1(\phi_1,\phi_2),\qquad\psi_2=\varphi_2(\phi_1,\phi_2)
\]
where $\varphi=(\varphi_1,\varphi_2)$ is birational. Hence 
\[
\psi\phi^{-1}=\big(\psi_0\circ\phi^{-1},\varphi_1(z_1,z_2),
\varphi_2(z_1,z_2)\big)
\]
belongs to $\mathscr{K}$; in other words $\phi$ and $\psi$ are in the same 
$\mathscr{K}$-orbit.

Assume now that $\psi=\kappa\phi$ where $\kappa$ denotes an element of $\mathscr{K}$.
Then the foliations defined by $Z_{\phi}$ and $Z_{\psi}$ coincide because they 
have the same set of first integrals. As a consequence $\alpha(\phi)=\alpha(\psi)$. 

Hence one can state:

\begin{thm}
The map $\alpha$ is a complete invariant of the left translation 
action of $\mathscr{K}$ on $\mathrm{Bir}(\mathbb{C}^3)_{\mathrm{c}(\omega)}$, that is 
for any $\phi$ and $\psi$ in $\mathrm{Bir}(\mathbb{C}^3)_{\mathrm{c}(\omega)}$ one has
$\alpha(\phi)=\alpha(\psi)$ if and only if $\psi\phi^{-1}$ belongs to 
$\mathscr{K}$.
\end{thm}

\begin{question}
Is the map $\alpha$ surjective ?
\end{question}

Let us consider the following differential equation
\begin{equation}\label{eq:eqdifford2}
y''=F(x,y,y')
\end{equation}
where $F$ denotes a rational function. Set $y'=u$, then 
\[
(\ref{eq:eqdifford2})\Leftrightarrow\left\{
\begin{array}{lll}
\frac{du}{dt}=F(x,y,u)\\
\frac{dy}{dt}=u\\
\frac{dx}{dt}=1
\end{array}
\right.
\]
So one can associate to (\ref{eq:eqdifford2}) the following vector field
\[
Z=F\frac{\partial}{\partial u}+u\frac{\partial}{\partial y}+\frac{\partial}{\partial x}.
\]

We say that (\ref{eq:eqdifford2}) is \textbf{\textit{rationally integrable}} if the 
vector field $Z$ has two first integrals $r_1$ and $r_2$ rationally independent: 
$\mathrm{d}r_1\wedge\mathrm{d}r_2\not\equiv 0$. 

For generic $\gamma$ and $\beta$
in $\mathbb{C}$ the differential equation $y''+\gamma y'+\beta y=0$ is not 
rationally integrable; as a consequence $-\gamma z_0-\beta z_2$ is not in the 
image of $\alpha$. The first Painlev\'e equation gives examples of polynomial 
of degree $2$ that does not belong to $\mathrm{im}\,\alpha$:

\begin{thm}[\cite{Casale}]
The equation $\mathcal{P}_1$ 
\[
y''=6y^2+x
\]
is not rationally integrable.
\end{thm}

If we come back with our notations it means that $6z_2^2-z_1$ is not in the
image of $\alpha$.

\begin{rem}
Indeed all generic Painlev\'e equations give rise to rational functions that 
do not belong to~$\mathrm{im}\,\alpha$.
\end{rem}

Nevertheless one can easily obtain examples of elements in the image of 
$\alpha$:

\begin{egs}
\begin{itemize}
\item If
$\phi=\left(\frac{z_0}{\beta},z_0+\beta z_1,z_2-\frac{z_0^2}{2\beta}\right)$
with $\beta\in\mathbb{C}^*$, then $\alpha(\phi)=\beta$.

\item If
\[
\phi=\big(z_0,z_1+P(z_0),z_2+Q(z_0)\big) 
\]
with $P$, $Q$ in $\mathbb{C}[z_0]$ such that $Q'(z_0)=-z_0P'(z_0)$, 
then $\alpha(\phi)=\frac{1}{P'(z_0)}$.

\item If 
\[
\phi=\big(-z_1,z_0+P(z_1),z_2+z_0z_1+Q(z_1)\big)
\]
with $P$, $Q$ in $\mathbb{C}[z_1]$ such that $Q'(z_1)=z_1P'(z_1)$ then
$\alpha(\phi)=P'(z_1)$.
\end{itemize}
\end{egs}

\medskip

Consider the left translation action of $\mathrm{Bir}(\mathbb{C}^3)_{\omega}$ on 
$\mathrm{Bir}(\mathbb{C}^3)_{\mathrm{c}(\omega)}$ defined by
\[
(\psi,\phi)\in \mathrm{Bir}(\mathbb{C}^3)_{\omega}\times 
\mathrm{Bir}(\mathbb{C}^3)_{\mathrm{c}(\omega)}\longrightarrow \psi\phi\in
\mathrm{Bir}(\mathbb{C}^3)_{\mathrm{c}(\omega)}.
\]

\begin{thm}
The map $V$ is a complete invariant of the left 
translation action of $\mathrm{Bir}(\mathbb{C}^3)_{\omega}$ on 
$\mathrm{Bir}(\mathbb{C}^3)_{\mathrm{c}(\omega)}$: for any 
$\phi$, $\psi$ in $\mathrm{Bir}(\mathbb{C}^3)_{\mathrm{c}(\omega)}$ one
has $V(\phi)=V(\psi)$ if and only if $\psi\phi^{-1}$ belongs to 
$\mathrm{Bir}(\mathbb{C}^3)_{\omega}$.
\end{thm}

\begin{proof}
Let $\phi$ be a contact birational map of $\mathbb{P}^3$. Obviously
$(f\phi)^*\omega=V(\phi)\omega$ for any $f\in\mathrm{Bir}(\mathbb{C}^3)_\omega$.

Let us now consider two contact birational maps $\phi$ and $\psi$ of
$\mathbb{P}^3$ such that $V=V(\phi)=V(\psi)$. On the one hand
\[
(\phi^{-1})^*\psi^*\omega=(\phi^{-1})^*V(\phi)\omega=V\circ\phi^{-1}\,(\phi^{-1})^*\omega
\]
and on the other hand composing $\phi^*\omega=V\omega$ by $(\phi^{-1})^*$ one gets
\[
\phi^*\omega=V\omega\Rightarrow (\phi^{-1})^*(\phi^*\omega)=(\phi^{-1})^*(V\omega)
\Rightarrow \omega=V\circ\phi^{-1}\,(\phi^{-1})^*\omega.
\]
As a consequence $(\phi^{-1})^*\psi^*\omega=\omega$, that is $\psi\phi^{-1}$ 
belongs to $\mathrm{Bir}(\mathbb{C}^3)_{\omega}$.
\end{proof}

\begin{pro}
If $\phi$ and $\psi$ are two contact birational maps of 
$\mathbb{P}^3$ such that $\alpha(\phi)=\alpha(\psi)$ and 
$V(\phi)=V(\psi)$, then $\psi\phi^{-1}$ belongs to 
\[
\left\{\left(\frac{z_0-b'(z_1)}{\nu'(z_1)},\nu(z_1),z_2+b(z_1)\right)\,\vert\, 
b\in\mathbb{C}(z_1),\,\nu\in\mathrm{PGL}(2;\mathbb{C})\right\}=\mathscr{K}
\cap\mathrm{Bir}(\mathbb{C}^3)_\omega.
\]
\end{pro}

\begin{proof}
Since both $\alpha(\phi)=\alpha(\psi)$ and $V(\phi)=V(\psi)$ 
the map $\psi\phi^{-1}$ is an element of 
$\mathrm{Bir}(\mathbb{C}^3)_{\omega}\cap\mathscr{K}$. One gets
the result from the descriptions of the Klein family and of 
$\mathrm{Bir}(\mathbb{C}^3)_{\omega}$ (Proposition \ref{Lem:chpinv}).
\end{proof}

Let us now give some examples of $V(\phi)$. 

\begin{egs}
\begin{itemize}
\item If $\phi$ belongs to $\mathscr{K}$, then
\[
V(\phi)=\frac{\frac{\partial \phi_1}{\partial z_1}\frac{\partial \phi_2}
{\partial z_2}-\frac{\partial\phi_1}{\partial z_2}\frac{\partial\phi_2}
{\partial z_1}}{\frac{\partial\phi_1}{\partial z_1}-z_0\frac{\partial \phi_1}
{\partial z_2}}.
\]

\item If 
\[
\phi=\left(\frac{1}{nz_0^{n-1}z_2+(n+1)z_0^n(z_1+1)},z_0^n\left(z_0+z_2+
z_0z_1\right),-z_0\right)
\]
with $n\in\mathbb{Z}$, then $V(\phi)=\frac{z_0}{(n+1)z_0z_1+nz_2+(n+1)z_0}$.

\item If
\[
\phi=\left(\frac{(z_1-z_0)^2}{2z_0z_1+2z_2-z_0^2},\frac{2z_2+z_0^2}{z_1-z_0},z_1-z_0\right),
\]
then $V(\phi)=\frac{2(z_0-z_1)}{z_0^2-2z_0z_1-2z_2}$.
\end{itemize}
\end{egs}

\begin{rem}
If $\phi$ belongs to $\mathrm{Bir}(\mathbb{C}^3)_{\mathrm{c}(\omega)}$, 
then $\phi^*\omega=V(\phi)\omega$ and 
$\phi^*(\omega\wedge\mathrm{d}\omega)=V(\phi)^2\omega\wedge\mathrm{d}\omega$
and $\det\mathrm{jac}\,\phi$ is a square. This gives some constraint on $V(\phi)$.
\end{rem}

As previously we can ask: is $V$ surjective ? The answer is no. Indeed let
us assume that there exists $\phi\in\mathrm{Bir}(\mathbb{C}^3)_{\mathrm{c}(\omega)}$
such that $V(\phi)=z_2$. Then 
$\phi_0\mathrm{d}\phi_0+\mathrm{d}\phi_2=z_0z_2\mathrm{d}z_1+\mathrm{d}\left(\frac{z_2^2}{2}\right)$
and 
$\mathrm{d}\phi_0\wedge\mathrm{d}\phi_1=\mathrm{d}(z_0z_2)\wedge\mathrm{d}z_1$.
Since the fibers of $(z_0z_2,z_1)$ are connected one can write $\phi_0$ as
$\varphi_0(z_0z_2,z_1)$ and $\phi_1$ as $\varphi_1(z_0z_2,z_1)$. Then 
$\phi^*\omega=V(\phi)\omega$ implies that 
$\phi_2-\frac{z_2^2}{2}=\varphi_2(z_0z_2,z_1)$. In other words
\[
\phi=\left(\varphi_0(z_0z_2,z_1),\varphi_1(z_0z_2,z_1),\varphi_2(z_0z_2,z_1)+
\frac{z_2^2}{2}\right).
\]
But $\phi\circ\left(\frac{z_0}{z_2},z_1,z_2\right)$ is clearly not birational 
so does $\phi$: contradiction.

\medskip

\subsection{Invariant forms and vector fields}

The next statement deals with flows in $\mathrm{Bir}(\mathbb{C}^3)_\omega$ 
(\emph{see} \cite{CerveauDeserti} for a definition).

\begin{pro}
Let $\phi_t$ be a flow in $\mathrm{Bir}(\mathbb{C}^3)_\omega$. Then 
$\phi_t$ has a first integral depending only on $(z_0,z_1)$ and with 
rational fibers.

In other words
\[
\phi_t=\big(\varphi_t(z_0,z_1),z_2+b_t(z_0,z_1)\big)
\]
where $\varphi_t$ belongs, up to conjugacy, to $\mathcal{J}$ and 
$b_t$ to $\mathbb{C}(z_0,z_1)$.
\end{pro}

\begin{proof}
Let $\chi$ be the infinitesimal generator of $\phi_t$, \emph{i.e.}
\[
\chi=\frac{\partial\phi_t}{\partial t}\Big\vert_{t=0}.
\]
By derivating $\phi_t^*\omega=\omega$ with respect to $t$ one gets that the 
Lie derivative $L_\chi\omega$ is zero. Set $\chi=\displaystyle
\sum_{i=0}^2\chi_i\frac{\partial}{\partial z_i}$, hence
\[
L_\chi\omega=\iota_\chi\mathrm{d}\omega+\mathrm{d}\iota_\chi\omega=\chi_0
\mathrm{d}z_1+z_0\mathrm{d}\chi_1+\mathrm{d}\chi_2
\]
and so
\[
L_\chi\omega=\left(z_0\frac{\partial\chi_1}{\partial z_0}+\frac{\partial\chi_2}
{\partial z_0}\right)\mathrm{d}z_0+\left(\chi_0+z_0\frac{\partial\chi_1}
{\partial z_1}+\frac{\partial\chi_2}{\partial z_1}\right)\mathrm{d}z_1+
\left(z_0\frac{\partial\chi_1}{\partial z_2}+\frac{\partial\chi_2}{\partial z_2}
\right)\mathrm{d}z_2.
\]
In particular $z_0\chi_1+\chi_2=\gamma(z_0,z_1)$, then $\chi_0+\frac{\partial}
{\partial z_1}(z_0\chi_1+\chi_2)=0$ so $\chi_0=-\frac{\partial\gamma}{\partial 
z_1}$ and finally $\chi_1=\frac{\partial\gamma}{\partial z_0}$.

If $\gamma$ is constant, then $\chi=\frac{\partial}{\partial z_2}$, that is 
$\phi_t=(z_0,z_1,z_2+\beta t)$ with $\beta\in\mathbb{C}$.

Let us now assume that $\gamma$ is non-constant; one has
\[
\chi=-\frac{\partial\gamma}{\partial z_1}\frac{\partial}{\partial z_0}+
\frac{\partial\gamma}{\partial z_0}\frac{\partial}{\partial z_1}+\left(\gamma
(z_0,z_1)-z_0\frac{\partial\gamma}{\partial z_0}\right)\frac{\partial}{\partial 
z_2}
\]
and $\gamma$ is a first integral of $\chi$. For all $t$
\[
\phi_t=\big(\phi_{0,t}(z_0,z_1),\phi_{1,t}(z_0,z_1),z_2+b_t(z_0,z_1)\big)
\]
and the function $\gamma$ is invariant by $\phi_t$ and as a consequence 
by the flow $\varphi_t$. The fibers of $\gamma$ in $\mathbb{C}^2$
(up to compactification/normalization) are rational or elliptic since they own 
a flow. As $\langle\varphi_t\rangle$ is uncountable they have to be rational 
(\cite{Cantat:these}) and up to conjugacy $\varphi_t$ belongs 
to~$\mathcal{J}$.
\end{proof}

The following examples contain many flows.

\begin{eg}\label{eg:autcont}
The elements of $\mathrm{Aut}(\mathbb{P}^3)_{\mathrm{c}(\omega)}$ can be written
\[
\big(\varepsilon z_0+\lambda,\beta z_1+\gamma,-\beta\lambda z_1+\varepsilon\beta 
z_2+\delta\big)
\]
with $\varepsilon$, $\beta$ in $\mathbb{C}^*$ and $\lambda$, $\gamma$, $\delta$ in 
$\mathbb{C}$. The group $\mathrm{Aut}(\mathbb{P}^3)_{\mathrm{c}(\omega)}$ acts 
transitively on $\mathbb{C}^3=\{z_3=1\}$.
\end{eg}

\begin{egs}\label{egs:trivialcomp}
\begin{enumerate}
\item[a)] For any $\varepsilon$, $\beta$, $\gamma$ and $\delta$ in $\mathbb{C}$ such that 
$\varepsilon\delta-\beta\gamma\not=0$, the map
\[
\left(\frac{(\gamma z_1+\delta)^2}{\varepsilon\delta-\beta\gamma}\,z_0,\frac{\varepsilon 
z_1+\beta}{\gamma z_1+\delta},z_2\right)
\]
belongs to $\mathrm{Bir}(\mathbb{C}^3)_\omega$. These maps form a group 
contained in $\mathrm{im}\,\mathcal{K}$ and isomorphic to $\mathrm{PGL}(2;\mathbb{C})$.
\smallskip

\item[b)] The birational maps given by
\begin{itemize}
\smallskip
\item $\big(z_0,z_1+\varphi(z_0),z_2+\psi(z_0)\big)$ with $z_0\varphi'(z_0)+\psi'(z_0)=0$,
\smallskip
\item $\big(z_0-\psi'(z_1),z_1,z_2+\psi(z_1)\big)$
\smallskip
\end{itemize}
belong to $\mathrm{Bir}(\mathbb{C}^3)_\omega$. Any of these families forms 
an abelian group.
\end{enumerate}
\end{egs}

The fact that an element of $\mathrm{Bir}(\mathbb{C}^3)_{\mathrm{c}(\omega)}$
preserves a vector field and the fact that it preserves a contact form are 
related:

\begin{pro}\label{pro:coli}
Let $\phi$ be a contact birational map of $\mathbb{P}^3$. 
There exist a contact form $\Theta$ colinear to $\omega$ such that 
$\phi^*\Theta=\Theta$ if and only if $V(\phi)$ can be written 
$\frac{U}{U\circ\phi}$ for some rational function $U$. In that case $\phi$ 
preserves the Reeb flow associated to $\Theta$, so a foliation by 
curves.
\end{pro}

\begin{proof}
Assume that such a $\Theta$ exists. On the one hand $\phi^*\omega=V(\phi)\omega$ 
and on the other hand $\Theta=U\omega$. Hence
\[
\phi^*\Theta=U\circ\phi \,\cdot\,\phi^*\omega=U\circ \phi\,\cdot\, V(\phi)\omega=
\frac{U\circ \phi}{U}\,\cdot\,V(\phi)\Theta
\]
and so if such $U$ exists, one has $V(\phi)=\frac{U}{U\circ\phi}$.

\smallskip

Reciprocally if $\phi\in\mathrm{Bir}(\mathbb{C}^3)_{\mathrm{c}(\omega)}
\smallsetminus\mathrm{Bir}(\mathbb{C}^3)_\omega$ satisfies 
$\phi^*\omega=\frac{U}{U\circ\phi}\omega$ for some rational function $U$, 
then $\phi^*\Theta=~\Theta$ where $\Theta=U\omega$.
\end{proof}

\begin{egs}
\begin{itemize}
\item First consider the Legendre involution 
$\mathcal{L}=(z_1,z_0,-z_2-z_0z_1)$. As we have seen $V(\mathcal{L})=-1$. One can check 
that $U~=~z_2+~\frac{z_0z_1}{2}$ suits.
\smallskip
\item For an element $\phi$ in $\mathrm{Aut}(\mathbb{P}^3)_{\mathrm{c}(\omega)}$
\[
\phi=\left(\varepsilon z_0+\lambda,\beta z_1+\gamma,-\beta\lambda z_1+
\varepsilon\beta z_2+\delta\right)
\]
with $\varepsilon$, $\beta$ in $\mathbb{C}^*$ and $\lambda$, $\gamma$, $\delta$ in 
$\mathbb{C}$ (Example \ref{eg:autcont}) we have $V(\phi)=\varepsilon\beta$. If
\[
U=\frac{\varepsilon\beta}{\varepsilon\beta z_0z_1+\varepsilon\gamma z_0+\beta\lambda 
z_1+\lambda\gamma}
\]
then $V(\phi)=\frac{U}{U\circ\phi}$.
\end{itemize}
\end{egs}

\begin{pro}\label{pro:cuicui}
Let $\phi$ be an element of $\mathrm{Bir}(\mathbb{C}^3)_{\mathrm{c}(\omega)}
\smallsetminus\mathrm{Bir}(\mathbb{C}^3)_\omega$. Assume that $\phi$ preserves 
a vector field $\chi$ non-tangent to $\omega$. Then $\phi$ preserves a contact form 
$\omega'$ colinear to $\omega$.
\end{pro}

\begin{rem}
Under these assumptions $\phi$ preserves the vector field $\chi$ and the Reeb 
vector field $Z$ associated to $\omega'$. With the previous notations if 
$f=z_0\chi_1+\chi_2$ and $g=z_0Z_1+Z_2$ one has 
$V(\phi)=\frac{f\circ\phi}{f}=\frac{g\circ\phi}{g}$. In particular if $H=f/g$ is non-constant, 
then $H$ is non-constant and invariant: $H\circ\phi=H$.
\end{rem}

\begin{proof}[Proof of Proposition \ref{pro:cuicui}]
Write $\chi$ as $\chi_0\frac{\partial}{\partial z_0}+\chi_1\frac{\partial}{\partial z_1}+
\chi_2\frac{\partial}{\partial z_2}$ and $\phi$ as $(\phi_0,\phi_1,\phi_2)$. Then 
$\phi_*\chi=\chi$ if and only if $\mathrm{d}\phi_i(\chi)=\chi_i\circ\phi$ for $i=0$, $1$ 
and $2$. Therefore $\phi^*\omega(\chi)=V(\phi)\omega(\chi)$ can be 
rewritten
\[
\phi_0\mathrm{d}\phi_1(\chi)+\mathrm{d}\phi_2(\chi)=\phi_0\chi_1\circ\phi+\chi_2\circ
\phi=V(\phi)(z_0\,\chi_1+\chi_2).
\]
The vector field $\chi$ is not tangent to $\omega$, {\it i.e.} $\omega(\chi)\not\equiv 0$ 
or in other words $z_0\chi_1+\chi_2\not\equiv 0$ and so
\[
V(\phi)=\frac{(z_0\chi_1+\chi_2)\circ\phi}{z_0\chi_1+\chi_2}.
\]
As a consequence $\phi$ preserves a contact form $\omega'$ colinear to $\omega$ (Proposition 
\ref{pro:coli}).
\end{proof}

\begin{rem}
Let $\phi\in\mathrm{Bir}(\mathbb{C}^3)_{\mathrm{c}(\omega)}\smallsetminus\mathrm{Bir}
(\mathbb{C}^3)_\omega$. Assume that there exists a vector field $\chi$ such that 
$\phi_*\chi=W\chi$. If $W$ can be written $\frac{U\circ\phi}{U}$, then $\phi$ preserves 
the vector field $Y=U\chi$. 
According to Proposition \ref{pro:cuicui} the map $\phi$ belongs to 
$\mathrm{Bir}(\mathbb{C}^3)_{\omega'}$ where $\omega'$ denotes a contact form 
colinear to $\omega$.
\end{rem}

\medskip

\subsection{Regular birational maps}

Let $\mathbf{e}_i$ be the point of $\mathbb{P}^3_\mathbb{C}$ whose all 
components are zero except the $i$-th.

Let us denote by $\mathcal{H}_\infty$ the hyperplane $z_3=0$. 
As~$\mathcal{H}_\infty$ is the unique invariant surface of~$\mathrm{c}(\omega)$ 
one has the following statement:

\begin{pro}\label{pro:regular}
The hyperplane $\mathcal{H}_\infty$ is either preserved, or blown down by any 
element of~$\mathrm{Bir}(\mathbb{C}^3)_{\mathrm{c}(\omega)}$.
\end{pro}

\begin{eg}
Let $\varphi$ be a birational map of the complex projective plane; 
$\mathcal{K}(\varphi)$ is polynomial if and only if 
$\varphi=\big(\beta z_1+\gamma,\delta z_2+P(z_1)\big)$ with $P\in\mathbb{C}[z_1]$; 
remark that such a $\varphi$ is a Jonqui\`eres polynomial automorphism. In that case
\[
\mathcal{K}(\varphi)=\left(\frac{1}{\beta}\left(\delta z_0-\frac{\partial P(z_1)}
{\partial z_1}\right),\beta z_1+\gamma,\delta z_2+P(z_1)\right).
\]
Note that $\deg P=1$ if and only if $\mathcal{K}(\varphi)$ is an automorphism of 
$\mathbb{P}^3$. If $\deg P>1$, then 
$\mathrm{Ind}\,\mathcal{K}(\varphi)=\big\{z_1=z_3=0\big\}$ and $\mathcal{H}_\infty$ 
is blown down onto $\mathbf{e}_3$.
\end{eg}

Proposition \ref{pro:regular} naturally implies the following definition. We 
say that $\phi\in\mathrm{Bir}(\mathbb{C}^3)_{\mathrm{c}(\omega)}$ is 
\textbf{\textit{regular at infinity}} if $\mathcal{H}_\infty$ is preserved by $\phi$ 
and if $\phi_{\vert \mathcal{H}_\infty}$ is birational. We denote by 
$\mathrm{Bir}(\mathbb{C}^3)^{\mathrm{reg}}_{\mathrm{c}(\omega)}$ $\big($resp. 
$\mathrm{Bir}(\mathbb{C}^3)^{\mathrm{reg}}_\omega$$\big)$ the set of regular 
maps at infinity that belong to $\mathrm{Bir}(\mathbb{C}^3)_{\mathrm{c}(\omega)}$ 
$\big($resp. $\mathrm{Bir}(\mathbb{C}^3)_\omega$$\big)$.

\begin{eg}\label{eg:aut}
Of course the elements of $\mathrm{Aut}(\mathbb{P}^3)_{\mathrm{c}(\omega)}$ 
$\big($Example \ref{eg:autcont}$\big)$ are regular at infinity.
\end{eg}

The contact structure is also given in homogeneous coordinates by the $1$-form
\[
\overline{\omega}=z_0z_3\mathrm{d}z_1+z_3^2\mathrm{d}z_2-(z_0z_1+z_2z_3)\mathrm{d}z_3.
\]
Let $\phi$ be an element of $\mathrm{Bir}(\mathbb{C}^3)^{\text{reg}}_{\mathrm{c}(\omega)}$; 
denote by $\overline{\phi}$ its homogeneization. Since $\phi^*\omega=V(\phi)\omega$ one 
has $\overline{\phi}\,\overline{\omega}=\overline{V(\phi)}\,\overline{\omega}$ where 
$\overline{V(\phi)}$ is a homogeneous polynomial. With these notations one can state:

\begin{lem}\label{lem:reginf}
Let $\phi$ be a contact birational map of $\mathbb{P}^3$. Assume 
that $\phi$ either preserves $\mathcal{H}_\infty$, or blows down~$\mathcal{H}_\infty$ 
onto a subset contained in $\mathcal{H}_\infty$.

The map $\phi$ is regular if and only if $\overline{V(\phi)}$ does not vanish identically 
on $\mathcal{H}_\infty$.
\end{lem}

\begin{proof}
Let us work in the affine chart $z_2=1$. On the one hand
\[
\overline{\omega}\wedge\mathrm{d}\overline{\omega}=-z_3^2\mathrm{d}z_0\wedge
\mathrm{d}z_1\wedge\mathrm{d}z_3
\]
and on the other hand
\[
\phi^*(\overline{\omega}\wedge\mathrm{d}\overline{\omega})=\overline{V(\phi)}^2\overline{\omega}\wedge
\mathrm{d}\overline\omega.
\]
Hence
\begin{equation}\label{eq:reg}
\overline{\phi}_3^2\,\,\det\mathrm{jac}\,\overline{\phi}=\overline{V(\phi)}^2z_3^2
\end{equation}
where $\overline{\phi}_3$ is the third component of $\overline{\phi}$ expressed 
in the affine chart $z_2=1$.

\medskip

Suppose that $\phi$ is regular. Let $p$ be a generic point of $\mathcal{H}_\infty$. 
As $\phi$ is regular, $\overline{\phi}_{\vert \mathcal{H}_\infty}$ is a local diffeomorphism 
at $p$. Since $\overline{\phi}$ is birational and $p$ is generic, $\overline{\phi}_{,p}$ 
is a local diffeomorphism. As a consequence $\det\mathrm{jac}\,\overline{\phi}$ is an unit 
at $p$; moreover the invariance of $\mathcal{H}_\infty$ by $\overline{\phi}$ implies 
that $\overline{\phi}_3=z_3u$ where $u$ is a unit. Therefore $\overline{V(\phi)}$ does 
not vanish at $p$.

\smallskip

Conversely assume that $\overline{V(\phi)}$ does not vanish identically on $\mathcal{H}_\infty$. 
As $\phi$ either preserves $\mathcal{H}_\infty$, or contracts~$\mathcal{H}_\infty$ onto a 
subset in $\mathcal{H}_\infty$, one can write $\overline{\phi}_3$ as $z_3P$. As a result
\[
(\ref{eq:reg})\quad\Leftrightarrow\quad P^2\,\,\det\mathrm{jac}\,\overline{\phi}=
\overline{V(\phi)}^2
\]
Since $\overline{V(\phi)}$ does not vanish the map $\phi$ is then regular at infinity.
\end{proof}

\begin{cor}\label{cor:descreg}
One has 
$\mathrm{Bir}(\mathbb{C}^3)^{\text{reg}}_\omega=\mathrm{Aut}(\mathbb{P}^3)_\omega$.
\end{cor}

\begin{proof}
Let $\phi$ be an element of $\mathrm{Bir}(\mathbb{C}^3)^{\text{reg}}_\omega$. 
From $\phi^*\omega=\omega$, one gets with the previous notations 
$\overline{\phi}^*\,\overline{\omega}=z_3^n\,\overline{\omega}$ for some integer $n$. 
Lemma \ref{lem:reginf} implies that $n=0$, that is 
$\overline{\phi}^*\overline{\omega}=\overline{\omega}$; then looking at the degree 
of the members of this equality one gets $\deg\phi=1$.
\end{proof}

\begin{eg}\label{eg:ecl}
The group $\mathrm{Bir}(\mathbb{C}^3)^{\mathrm{reg}}_{\mathrm{c}(\omega)}$ contains blow-ups 
in restriction to $\mathcal{H}_\infty$. Indeed let us look at $\omega$ in the affine 
chart $z_2=1$ and consider the birational map $\phi$ given in $z_2=1$ by
\[
\phi=\big(z_0,z_0z_1-z_3,z_0z_3\big).
\]
Since $(\phi^n)^*\omega=z_0^{-n}\omega$,  
$\phi^n\in\mathrm{Bir}(\mathbb{C}^3)_{\mathrm{c}(\omega)}^{\mathrm{reg}}\smallsetminus\mathrm{Bir}(\mathbb{C}^3)_\omega$ 
for any $n\not=0$; in restriction to $\mathcal{H}_\infty$ the map $\phi^n$ 
coincides with $(z_0,z_1z_0^n)$.

Let us note that $\mathrm{Ind}\,\phi^n=\{\mathbf{e}_1\}\cup(z_0=z_2=0)$, that $z_0=0$ 
is contracted by $\phi$ onto $(z_0=z_2=0)$ and $z_2=0$ onto $(z_0=z_3=0)$. Besides
 $\mathrm{Ind}\,\phi^{-n}=\{z_0=z_2=0\}\cup\{z_0=z_3=0\}$, $(z_0=0)$ is 
blown down by $\phi^{-1}$ onto~$\mathbf{e}_2$ and $(z_2=0)$ onto $\mathbf{e}_1$.
\end{eg}

\begin{rem}
The group generated by Examples \ref{eg:aut} and \ref{eg:ecl} is in restriction to 
$\mathcal{H}_\infty$ and in the affine chart $z_2=1$
\[
\langle \left(\frac{\gamma z_0}{\beta z_1+\lambda},\frac{\lambda z_1}{\gamma(
\beta z_1+\lambda)}\right),\,(z_0,z_0z_1)\,\vert\,\gamma,\,\beta\in\mathbb{C}^*,
\,\lambda\in\mathbb{C}\rangle;
\]
it is of course a subgroup of 
$\mathrm{Bir}(\mathbb{C}^3)^{\mathrm{reg}}_{\mathrm{c}(\omega)}$.
\end{rem}

\begin{question}
Does this group coincide with 
$\mathrm{Bir}(\mathbb{C}^3)^{\mathrm{reg}}_{\mathrm{c}(\omega)}$ ?
\end{question}

\begin{egs}
\begin{itemize}
\item[a)] If $\phi$ is either a monomial map (\emph{i.e.} a map of the form
$(z_1^pz_2^q,z_1^rz_2^s)$ with $\left[
\begin{array}{cc}
p & q \\
r & s\\
\end{array}
\right]$ in $\mathrm{GL}(2;\mathbb{Z})$), or a 
non-linear polynomial automorphism, 
or a Jonqui\`eres map, then $\mathcal{K}(\phi)$ is not regular at infinity.
\smallskip

\item[b)] The map of order $5$ given by 
$\left(-\frac{z_2+1+z_0z_1}{z_0z_1^2},z_2,\frac{z_2+1}{z_1}\right)$, 
the map $\left(\frac{z_0}{(z_2+1)^2},z_1,\frac{z_2}{z_2+1}\right)$ and Examples 
\ref{egs:trivialcomp} a) are non-regular at infinity.
\smallskip

\item[c)] Any map of the form
\[
\left(\frac{1}{z_0}-f'(z_2),z_2,z_1+f(z_2)\right)
\]
is in 
$\mathrm{Bir}(\mathbb{C}^3)_{\mathrm{c}(\omega)}\smallsetminus\mathrm{Bir}(\mathbb{C}^3)_\omega$ 
and is not regular at infinity.
\smallskip
\item[d)] Elements of the Legendre family are not regular at infinity.
\end{itemize}
\end{egs}

\medskip

\subsection{Exact birational maps}

\smallskip

Recall that an element $\phi$ of $\mathrm{Bir}(\mathbb{C}^2)_\eta$ is 
\textbf{\textit{exact}} if it can be lifted via $\varsigma$ to 
$\mathrm{Bir}(\mathbb{C}^3)_\omega$, or equivalently if it belongs to 
$\mathrm{im}\,\varsigma$. The following statement allows to determine such maps.

\begin{thm}\label{Thm:critere}
A map $\big(\phi_0(z_0,z_1),\phi_1(z_0,z_1)\big)\in\mathrm{Bir}(\mathbb{C}^2)_\eta$ 
is exact if and only if the closed form $\phi_0\mathrm{d}\phi_1-z_0\mathrm{d}z_1$ 
has trivial residues. In that case 
$\phi_0\mathrm{d}\phi_1-z_0\mathrm{d}z_1=-\mathrm{d}b$ with $b\in\mathbb{C}(z_0,z_1)$
and
\[
\phi=\big(\phi_0(z_0,z_1),\phi_1(z_0,z_1),z_2+b(z_0,z_1)\big)
\]
belongs to $\mathrm{Bir}(\mathbb{C}^3)_\omega$.
\end{thm}

\begin{proof}
Remark that $\phi=\big(\phi_0(z_0,z_1),\phi_1(z_0,z_1),z_2+b(z_0,z_1)\big)$ belongs 
to $\mathrm{Bir}(\mathbb{C}^3)_\omega$ if and only if
\[
\phi_0\mathrm{d}\phi_1-z_0\mathrm{d}z_1=-\mathrm{d}b;
\]
in other words $\phi_0\mathrm{d}\phi_1-z_0\mathrm{d}z_1$ is not only 
a closed rational $1$-form but also an exact one. Recall that a closed 
rational $1$-form $\Theta$ can be written (\cite{CerveauMattei})
\[
\Theta=\sum_i\lambda_i\frac{\mathrm{d}f_i}{f_i}+\mathrm{d}g
\]
where the $\lambda_i$ are complex numbers and the $f_i$'s and $g$ are rational. 
The $1$-form $\Theta$ is exact ({\it i.e.} the differential of a rational function) 
if $\lambda_i=0$ for all $i$, that is if the residues of $\Theta$ are trivial.
\end{proof}

\begin{eg}
The set
\[
\left\{\left(A(z_0),\frac{z_1}{A'(z_0)}\right)\,\vert\, A\in\mathrm{PGL}(2;
\mathbb{C})\right\}
\]
is a subgroup of exact maps isomorphic to $\mathrm{PGL}(2;\mathbb{C})$; it is 
a direct consequence of Theorem \ref{Thm:critere}.
\end{eg}

An other direct consequence of Theorem \ref{Thm:critere} is the following 
statement:

\begin{cor}
The maps $\phi=(\phi_0,\phi_1)$ of $\mathrm{Bir}(\mathbb{C}^2)_\eta$ 
such that $\phi_0\mathrm{d}\phi_1-z_0\mathrm{d}z_1$ has trivial residues form 
a group. 
\end{cor}

\smallskip

Let us deal with exact birational involutions.

Bertini gives a classification of birational involutions (\cite{Bertini}): 
a non-trivial birational involution is conjugate to either a Jonqui\`eres 
involution of degree $\geq 2$, or a Bertini involution, or a Geiser 
involution. More recently Bayle and Beauville precise it 
(\cite{BayleBeauville}); the map which associates to a birational involution 
of $\mathbb{P}^2$ its normalized fixed curve establishes a one-to-one 
correspondence between:
\smallskip
\begin{itemize}
\item conjugacy classes of Jonqui\`eres involutions of degree $d$ and 
isomorphism classes of hyperelliptic curves of genus $d-2$ ($d\geq 3$);
\smallskip
\item conjugacy classes of Geiser involutions and isomorphism classes 
of non-hyperelliptic curves of genus~$3$;
\smallskip
\item conjugacy classes of Bertini involutions and isomorphism classes of 
non-hyperelliptic curves of genus $4$ whose canonical model lies on a singular 
quadric.
\smallskip
\end{itemize}

Besides the Jonqui\`eres involutions of degree $2$ form one conjugacy class.

\begin{pro}\label{Pro:inv}
Let $\mathcal{I}\in\mathrm{Bir}(\mathbb{P}^2)$ be a birational 
involution. If $\mathcal{I}$ is conjugate to either a Geiser involution, 
or a Bertini involution, or a Jonqui\`eres involution of degree 
$\geq 3$, then $\mathcal{I}$ does not belong to 
$\mathrm{Bir}(\mathbb{C}^2)_\eta$.

Hence the only involutions in $\mathrm{Bir}(\mathbb{C}^2)_\eta$ are 
birationally conjugate to $(-z_0,-z_1)$. Some of them can not be lifted.
\end{pro}

\begin{proof}
Let us consider such an involution, then the set of fixed points contains a curve 
$\Gamma$ of genus $>0$ and thus it is not contained in the line at infinity. The 
jacobian determinant of $\mathcal{I}$ at a fixed point of $\Gamma$ is~$-1$ 
hence~$\mathcal{I}$ does not preserve $\eta$.

Contrary to the polynomial case (Proposition \ref{pro:periodic})
$\mathrm{Bir}(\mathbb{C}^2)_\eta$ 
contains periodic elements that are non-exact. Consider the map 
$(\phi_0(z_0,z_1),\phi_1(z_0,z_1))$ where
\[
\phi_0(z_0,z_1)=-z_0+\frac{1}{z_1^2-1},\qquad \phi_1(z_0,z_1)=-z_1;
\]
it is a birational involution that preserves $\eta$. Furthermore the 
$1$-form $\phi_0\mathrm{d}\phi_1-z_0\mathrm{d}z_1$ has non-trivial residues 
and so is not exact (Theorem \ref{Thm:critere}).
\end{proof}

\smallskip

We will now focus on quadratic exact birational maps.

Any birational map of $\mathbb{P}^2$ can be written as a composition of 
birational maps of degree $\leq 2$ (\emph{see} for instance \cite{AlberichCarraminana}). 
The three following maps are birational and of degree $2$
\begin{align*}
&\sigma\colon\mathbb{P}^2\dashrightarrow\mathbb{P}^2 && (z_0:z_1:z_2)
\dashrightarrow (z_1z_2:z_0z_2:z_0z_1)\\
&\rho\colon\mathbb{P}^2\dashrightarrow\mathbb{P}^2 && (z_0:z_1:z_2)
\dashrightarrow (z_0z_2:z_0z_1:z_2^2)\\
&\tau\colon\mathbb{P}^2\dashrightarrow\mathbb{P}^2 && (z_0:z_1:z_2)
\dashrightarrow (z_0z_2+z_1^2:z_1z_2:z_2^2)
\end{align*}
Denote by~$\mathrm{\mathring{B}ir}_2(\mathbb{P}^2)$ the set of birational maps 
of $\mathbb{P}^2$ of degree $2$ exactly; for 
any~$\phi\in\mathrm{Bir}(\mathbb{P}^2)$ set
\[
\mathcal{O}(\phi)=\big\{\mathfrak{g}\,\phi\,\mathfrak{h}^{-1}\,\vert\, \mathfrak{g},\,
\mathfrak{h}\in\mathrm{Aut}(\mathbb{P}^2)\big\}
\]
one has (\cite{CerveauDeserti})
\[
\mathrm{\mathring{B}ir}_2(\mathbb{P}^2)=\mathcal{O}(\sigma)\cup\mathcal{O}(\rho)
\cup\mathcal{O}(\tau).
\]

\smallskip

Let us now describe the quadratic birational maps that preserve $\eta$; note that 
$\tau$ preserves $\eta$. Consider $\Upsilon$ the set of pairs 
$(\mathfrak{g}(\gamma),\mathfrak{g}(\beta))$ where
\[
\mathfrak{g}(\beta)=\left(\frac{\beta_0z_0+\beta_1z_1+\beta_2}{\beta_6z_0+\beta_7z_1+
\beta_8},\frac{\beta_3z_0+\beta_4z_1+\beta_5}{\beta_6z_0+\beta_7z_1+\beta_8}\right)
\]
in $\mathrm{Aut}(\mathbb{P}^2)\times\mathrm{Aut}(\mathbb{P}^2)$ 
such that
\[
\gamma_6=0,\quad\gamma_7\beta_3=0,\quad\gamma_7\beta_4=0,
\quad\det\mathfrak{g}\,\det\mathfrak{h}=\big(\gamma_7\beta_5+\gamma_8\big)^3.
\]

\smallskip

\begin{pro}\label{Pro:quadpre}
A quadratic birational map that preserves $\eta$ belongs to 
$\mathcal{O}(\tau)$.

More precisely a birational map belongs to 
$\mathrm{\mathring{B}ir}_2(\mathbb{P}^2)\cap\mathrm{Bir}(\mathbb{C}^2)_\eta$ 
if and only if it can be written $\mathfrak{g}\,(z_0+z_1^2,z_1)\,\mathfrak{h}$ 
with $(\mathfrak{g},\mathfrak{h})$ in $\Upsilon$.
\end{pro}

\begin{proof}
Let $\psi$ be in 
$\mathrm{Bir}(\mathbb{C}^2)_\eta\cap\mathrm{\mathring{B}ir}_2(\mathbb{P}^2)$; 
it is sufficient to prove that $\psi\not\in\mathcal{O}(\sigma)\cup\mathcal{O}(\rho)$.

Assume by contradiction that $\psi$ belongs to $\mathcal{O}(\sigma)$, {\it i.e.} 
$\psi=\mathfrak{g}\sigma\mathfrak{h}$ with
$\mathfrak{g}=\mathfrak{g}(\gamma)$, $\mathfrak{h}^{-1}=\mathfrak{g}(\beta)$.
One can rewrite $\psi^*\eta=\eta$ as $\sigma^*\mathfrak{g}^*\eta=\mathfrak{h}^*\eta$; 
this last one relation is equivalent in the affine chart $z_3=1$ to
\begin{equation}\label{eq:quad}
\frac{(\det\mathfrak{g})\, z_0z_1}{\big(\gamma_6z_1+\gamma_7z_0+\gamma_8z_0z_1\big)^3}
\,\eta=\frac{\det\mathfrak{h}}{\big(\beta_6z_0+\beta_7z_1+\beta_8\big)^3}\,\eta
\end{equation}
the coefficients $\gamma_6$ and $\gamma_7$ have thus to be zero and $(\ref{eq:quad})$
is equivalent to 
\[
\frac{\det\mathfrak{g}}{\gamma_8^3z_0^2z_1^2}
\,\eta=\frac{\det\mathfrak{h}}{\big(\beta_6z_0+\beta_7z_1+\beta_8\big)^3}\,\eta
\]
and this equality never holds.

\medskip

A similar argument allows to exclude the case: $\psi\in\mathcal{O}(\rho)$. This proves 
the first assertion.

\medskip

Let us consider $\psi=\mathfrak{g}\,\tau\,\mathfrak{h}$ in 
$\mathrm{\mathring{B}ir}_2(\mathbb{P}^2)\cap\mathrm{Bir}(\mathbb{C}^2)_\eta$ 
with $\mathfrak{g}=\mathfrak{g}(\gamma)$ and $\mathfrak{h}=\mathfrak{g}(\beta)$.
The $1$-form $\eta$ has a line of poles of order $3$ at infinity so does $\psi^*\eta$ 
and so does $(z_0+z_1^2,z_1)^*\mathfrak{g}^*\eta$. But
\[
(z_0+z_1^2,z_1)^*\mathfrak{g}^*\eta=\frac{\det\mathfrak{g}}{\big(\gamma_6(z_0+z_1^2)+
\gamma_7z_1+\gamma_8\big)^3}\,\eta
\]
therefore $\gamma_6$ has to be $0$. This implies that
\[
\psi^*\eta=\frac{\det\mathfrak{g}\,\det\mathfrak{h}}{\big(\gamma_7(\beta_3z_0+\beta_4z_1
+\beta_5)+\gamma_8\big)^3}\,\eta
\]
as a consequence $\psi^*\eta=\eta$ if and only if
\[
\gamma_6=0,\quad\gamma_7\beta_3=0,\quad\gamma_7\beta_4=0,
\quad\det\mathfrak{g}\,\det\mathfrak{h}=\big(\gamma_7\beta_5+\gamma_8\big)^3.
\]
\end{proof}

\begin{thm}\label{Thm:quadexact}
A generic element of 
$\mathrm{\mathring{B}ir}_2(\mathbb{P}^2)\cap\mathrm{Bir}(\mathbb{C}^2)_\eta$ 
is not exact.

\smallskip

In fact there exists a non-empty Zariski open subset $\widetilde{\Upsilon}$ 
of $\Upsilon$ such that no element of
\[
\big\{\mathfrak{g}(\gamma)\,\tau\,\mathfrak{g}(\beta)\,\vert\,(\mathfrak{g}(\gamma),
\mathfrak{g}(\beta))\in\widetilde{\Upsilon}\big\}
\]
is exact.
\end{thm}

\begin{proof}
It is sufficient to exhibit a non-exact element. Let us recall that the birational map 
$\phi=(\phi_0,\phi_1)$ belongs to 
$\mathrm{\mathring{B}ir}_2(\mathbb{P}^2)\cap\mathrm{Bir}(\mathbb{C}^2)_\eta$ 
if and only if it can be written as 
$\mathfrak{g}(\gamma)\,\tau\,\mathfrak{g}(\beta)$ with 
$(\mathfrak{g}(\gamma),\mathfrak{g}(\beta))$ in $\Upsilon$ (Proposition~\ref{Pro:quadpre}).

\medskip

If we consider the special case $\gamma_i=\beta_i=0$ for any $i\in\{1,\,2,\,3,\,4,\,6,\,8\}$,
$\gamma_5=\gamma_7$ and $\gamma_0=\frac{\gamma_7\beta_5^2}{\beta_0\beta_7}$
then
\[
z_0\mathrm{d}z_1-\phi_0\mathrm{d}\phi_1=-\frac{\beta_5^2\mathrm{d}z_1}{\beta_0\beta_7 z_1}
\]
But $\det\mathfrak{g}(\beta)\not=0$ so $\beta_5\not=0$ and $\phi$ can not be 
lifted to $\mathrm{Bir}(\mathbb{C}^3)_\omega$.

\medskip

The set $\Upsilon$ is rational hence irreducible, this yields the result.
\end{proof}

\smallskip

Let us end this section with examples of exact maps.

\begin{pro}
Let $\varphi$ be an automorphism of $\mathbb{P}^2$; the map $\varphi$ is 
exact if and only if $\varphi$ is affine in the affine chart $z_2=1$ and preserves $\eta$, 
that is
\[
\varphi=\big(\delta_0z_0+\beta_0z_1+\gamma_0,\delta_1z_0+\beta_1z_1+\gamma_1\big)
\]
with $\delta_i$, $\beta_i$, $\gamma_i$ in $\mathbb{C}$ such that 
$\delta_0\beta_1-\delta_1\beta_0=1$.
\end{pro}

\begin{proof}
The form $\eta$ has a pole at infinity so if $\varphi\in\mathrm{Aut}(\mathbb{P}^2)$ 
preserves $\eta$, it preserves the pole. Hence $\varphi$ belongs to $\mathrm{Aff}_2$, 
so in particular to $\mathrm{Aut}(\mathbb{C}^2)_\eta$ and then $\varphi$ is exact.
\end{proof}

We will now consider the subgroup of $\mathrm{Bir}(\mathbb{C}^2)_\eta$ 
that preserves the fibration $z_0z_1=$ cst fiberwise. The following statement 
says that this subgroup is not isomorphic to the subgroup of 
$\mathrm{Bir}(\mathbb{C}^2)_\eta$ that preserves $z_1=$ cst fiberwise.

\begin{pro}\label{Pro:abgr}
The set
\[
\Lambda=\left\{\left(z_0\,a(z_0z_1),\frac{z_1}{a(z_0z_1)}\right)\,\vert\, 
a\in\mathbb{C}(t)\right\}
\]
is a subgroup isomorphic to the uncountable abelian subgroup 
$\big\{(a(z_1)z_0,z_1)\,\vert\,a\in\mathbb{C}(z_1)^*\big\}$ and is contained 
in~$\mathrm{Bir}(\mathbb{C}^2)_\eta$.

Any birational map of the form 
$\left(z_0\,a(z_0,z_1),\frac{z_1}{a(z_0,z_1)}\right)$ that preserves $\eta$ 
belongs to~$\Lambda$.

A generic element of $\Lambda$ is in $\mathrm{Bir}(\mathbb{C}^2)_\eta$ 
but not in $\mathrm{im}\,\varsigma$. More precisely 
$\left(z_0\,a(z_0z_1),\frac{z_1}{a(z_0z_1)}\right)\in\Lambda$ is exact if and 
only if $a$ is a monomial.

If $a$ is a monomial, \emph{i.e.} $a(z_0z_1)=cz_0^\mu z_1^\mu$ with 
$c\in\mathbb{C}^*$ and $\mu\in\mathbb{Z}$, then the $\varsigma$-lifted maps 
are
\[
\left(z_0\,cz_0^\mu z_1^\mu,\frac{z_1}{cz_0^\mu z_1^\mu},z_2-\mu z_0z_1 + \beta
\right),\qquad\beta\in\mathbb{C}
\]
These maps form a subgroup of $\mathrm{Bir}(\mathbb{C}^3)_\omega$ 
isomorphic to $\mathbb{C}\times\mathbb{C}^*\times\mathbb{Z}$.
\end{pro}

\begin{proof}
The first assertion follows from
\[
\left(z_0\,a(z_0z_1),\frac{z_1}{a(z_0z_1)}\right)=(z_0,z_0z_1)^{-1}(z_0\,a(z_1),
z_1)(z_0,z_0z_1)
\]
A direct computation shows that 
$\Lambda\subset\mathrm{Bir}(\mathbb{C}^2)_\eta$.

\bigskip

A birational map $\left(z_0\,a(z_0,z_1),\frac{z_1}{a(z_0,z_1)}\right)$ preserves 
$\eta$ if and only if
\[
\left(z_0\,\frac{\partial}{\partial z_0}-z_1\,\frac{\partial}{\partial z_1}
\right)(a)=0
\]
that is, if and only if $a=a(z_0z_1)$.

\bigskip

Let us consider 
$\phi=(\phi_0,\phi_1)=\left(z_0\,a(z_0z_1),\frac{z_1}{a(z_0z_1)}\right)$ an 
element of $\Lambda$; then
\[
\phi_0\mathrm{d}\phi_1-z_0\mathrm{d}z_1=t\frac{a'(t)}{a(t)}\mathrm{d}t
\]
with $t=z_0z_1$. Let us write $a$ as follows:
\[
a(t)=\prod_{i=1}^n(t-t_i)^{\mu_i}
\]
then
\[
t\frac{a'(t)}{a(t)}\mathrm{d}t=t\,\sum_{i=1}^n\frac{\mu_i}{t-t_i}\mathrm{d}t
\]
and the residues of this $1$-form are trivial if and only if $a$ is monomial, 
{\it i.e.} $a(t)=c\,t^{\mu}$ where $c\in\mathbb{C}^*$ and $\mu\in\mathbb{Z}$.
\end{proof}

We can determine $\mathcal{J}\cap\mathrm{Bir}(\mathbb{C}^2)_\eta$ and 
the exact maps in $\mathcal{J}\cap\mathrm{Bir}(\mathbb{C}^2)_\eta$.

\begin{pro}
A Jonqui\`eres map of $\mathbb{P}^2$ preserves $\eta$ 
if and only if it can be written as follows
\[
\left(\frac{(\gamma z_1+\delta)^2}{\varepsilon\delta-\beta\gamma}\,z_0+r(z_1),
\frac{\varepsilon z_1+\beta}{\gamma z_1+\delta}\right)
\]
where $r$ belongs to $\mathbb{C}(z_1)$ and $\left[\begin{array}{cc}
\varepsilon & \beta \\
\gamma & \delta
\end{array}\right]$ to $\mathrm{PGL}(2;\mathbb{C})$.

\smallskip

Furthermore it is exact if it has the following form
\[
\left(\frac{(\gamma z_1+\delta)^2}{\varepsilon\delta-\beta\gamma}\,z_0+P(z_1)(\gamma 
z_1+\delta)^2,\frac{\varepsilon z_1+\beta}{\gamma z_1+\delta}\right)
\]
where $P$ denotes an element of $\mathbb{C}[z_1]$.
\end{pro}

Let us now look at monomial maps that belong to 
$\mathrm{Bir}(\mathbb{C}^2)_\eta$ and those who are exact.

\begin{pro}
A monomial map belongs to $\mathrm{Bir}(\mathbb{C}^2)_\eta$ 
if and only if it can be written either
\begin{equation}\label{mon1}
\left(\gamma\,z_0^{p}z_1^{p-1},\frac{1}{\gamma}\,z_0^{1-p}z_1^{2-p}\right)
\end{equation}
or
\begin{equation}\label{mon2}
\left(\gamma\,z_0^{p}z_1^{p+1},-\frac{1}{\gamma}\,z_0^{1-p}z_1^{-p}\right)
\end{equation}
with $\gamma$ in $\mathbb{C}^*$ and $p$ in $\mathbb{Z}$.

Furthermore any monomial map of $\mathrm{Bir}(\mathbb{C}^2)_\eta$ 
is exact.

The $\varsigma$-lifts of a map of type $($\ref{mon1}$)$ are
\[
\Big(\gamma z_0^pz_1^{p-1},\frac{1}{\gamma}z_0^{1-p}z_1^{2-p},z_2+(p-1)z_0
z_1+\beta\Big)\qquad\qquad \beta\in\mathbb{C}
\]
similarly the $\varsigma$-lifts of a map of type $($\ref{mon2}$)$ are
\[
\Big(\gamma z_0^pz_1^{p+1},-\frac{1}{\gamma}z_0^{1-p}z_1^{-p},z_2+(1-p)z_0
z_1+\beta'\Big)\qquad\qquad\beta'\in\mathbb{C}
\]
\end{pro}

\begin{rems}
\begin{itemize}
\item Both maps of type $($\ref{mon1}$)$ and of type $($\ref{mon2}$)$ 
preserve $(z_0z_1)^2=$ cst.

\item Maps of type $($\ref{mon1}$)$ form a group $\mathrm{G}_1$. Note that 
the matrices 
$\left[
\begin{array}{cc}
p & p-1 \\
1-p & 2-p
\end{array}
\right]$ 
are in $\mathrm{SL}(2;\mathbb{Z})$; they are stochastic up to transposition 
and have trace equal to $2$. The group
\[
\left\{\left[
\begin{array}{cc}
p & p-1 \\
1-p & 2-p
\end{array}
\right]\,\Big\vert\, p\in\mathbb{Z}\right\}
\]
is isomophic to $\mathbb{Z}$. As a consequence $\mathrm{G}_1$ is isomorphic 
to $\mathbb{C}^*\times\mathbb{Z}$.

\smallskip

The maps of type $($\ref{mon2}$)$ don't form a group. The corresponding 
matrices 
$\left[
\begin{array}{cc}
p & p+1 \\
1-p & -p
\end{array}
\right]$ 
have determinant $-1$, trace $0$ and are stochastic up to transposition.

\smallskip

But the union of the maps of type $($\ref{mon1}$)$ or $($\ref{mon2}$)$ is a 
group which is a double extension of $\mathbb{C}^*\times\mathbb{Z}$.
\end{itemize}
\end{rems}

\medskip

\subsection{Indeterminacy and exceptional sets}

As we have seen if $\phi$ is a contact map, then $\mathcal{H}_\infty$ is either 
preserved by $\phi$, or blown down by $\phi$ (Proposition~\ref{pro:regular}). 
In case it is blown down, $\mathcal{H}_\infty$ can be blown down onto a point 
or onto a curve; in this last eventuality $\mathcal{H}_\infty$ can be contracted 
onto a curve contained in $\mathcal{H}_\infty$ (take for instance 
$\phi=\mathcal{K}(z_1,z_1z_2)$). Note also that $\mathcal{H}_\infty$ can be 
contracted onto a curve not contained in $\mathcal{H}_\infty$: the map 
$\mathcal{K}\left(\frac{z_1}{z_2},\frac{1}{z_2}\right)$ blows down 
$\mathcal{H}_\infty$ onto the legendrian curve $z_0=z_2=0$. We will see that 
this is a general case and for any contracted surface:

\begin{pro}\label{pro:contracourbe}
Let $\phi$ be a contact birational map of $\mathbb{P}^3$. Assume 
that $\phi$ blows down a surface $\mathcal{S}$ onto a curve $\mathcal{C}$. Then
\begin{itemize}
\smallskip
\item either $\mathcal{C}$ is contained in $\mathcal{H}_\infty$,
\smallskip
\item or $\mathcal{C}$ is an algebraic legendrian curve.
\end{itemize}
\end{pro}

\begin{cor}
Let $\phi$ be a contact birational map of $\mathbb{P}^3$. If 
$\mathcal{C}$ is a curve not contained in $\mathcal{H}_\infty$ and
blown-up by $\phi$ on a surface distinct from $\mathcal{H}_\infty$,
then $\mathcal{C}$ is a legendrian curve.
\end{cor}

Let us now give an example of maps of finite order that illustrates 
Proposition \ref{pro:contrapoint}.

\begin{eg}\label{eg:ordrefini}
Start with the birational map 
$\varphi=\left(z_2,\frac{z_2+1}{z_1}\right)$ of order $5$. The map 
$\mathcal{K}(\varphi)=\left(-\frac{z_2+1+z_0z_1}{z_0z_1^2},z_2,\frac{z_2+1}{z_1}\right)$ 
blows down $z_2=-z_3$ onto the legendrian curve $(z_2=z_1+z_3=0)$;
\end{eg}

\begin{proof}[Proof of Proposition \ref{pro:contracourbe}]
We will distinguish the cases $\mathcal{S}=\mathcal{H}_\infty$ and  
$\mathcal{S}\not=\mathcal{H}_\infty$.

Let us start with the eventuality $\mathcal{S}=\mathcal{H}_\infty$. Suppose 
that $\mathcal{C}$ is not contained in $\mathcal{H}_\infty$. Note that 
$\phi_{\vert\mathcal{H}_\infty\smallsetminus\mathrm{Ind}\,\phi}$ is holomorphic of 
rank $\leq 1$. If $p$ belongs to $\mathcal{C}\smallsetminus\mathrm{Ind}\,\phi$, 
then $\phi^{-1}(p)$ is a curve contained in $\mathcal{H}_\infty$; there 
exists a curve $\mathcal{C}'$ transverse to
\[
\big\{\phi^{-1}(p) \,\vert\,p\in\mathcal{C}\smallsetminus\mathrm{Ind}\,\phi
\big\}
\]
contained in $\mathcal{H}_\infty$ and such that 
$\phi(\mathcal{C}')=\mathcal{C}$. 
Consider a parametrization $s$ of $\mathcal{C}'$; then $t=\phi\circ s$ is a 
parametrization of $\mathcal{C}$ and
\[
t^*\omega=(\phi\circ s)^*\omega=s^*\phi^*\omega=s^*V(\phi)\omega=V(\phi)\circ s\cdot 
s^*\omega=0.
\]

\medskip

Assume now that $\mathcal{S}\not=\mathcal{H}_\infty$ and 
$\mathcal{C}\not\subset \mathcal{H}_\infty$. Set $\mathcal{C}=\phi(\mathcal{S})$. 
Let us consider a generic point $p$ of $\mathcal{S}$. The germ $\phi_{,p}$ 
is holomorphic and $\phi(p)\in\mathcal{C}$ does not belong to 
$\mathcal{H}_\infty$. In particular the $3$-form 
$\phi^*\omega\wedge\mathrm{d}\omega$ is thus holomorphic at~$p$; 
in fact $V(\phi)_{,p}$ is holomorphic and as we have seen
\[
\phi^*\omega\wedge\mathrm{d}\omega=V(\phi)^2\omega\wedge\mathrm{d}\omega.
\]
Since $\mathcal{S}$ is blown down by  $\phi$, the jacobian determinant 
of $\phi$ is identically zero on $\mathcal{S}$ and then $V(\phi)$ vanishes 
on~$\mathcal{S}$.

Assume that $\mathcal{C}$ is not a legendrian curve, then the restriction 
of $\omega$ to $\mathcal{C}$ in a neighborhood of $\phi(p)$ defines a 
$1$-form $\Theta$ on $\mathcal{C}$ without zero (let us recall that $p$ 
is generic). As the restriction
\[
\phi_{,p\vert_{\mathcal{S},p}}\colon\mathcal{S}_{,p}\to\mathcal{C}_{,\phi(p)}
\]
is locally a submersion, $\phi_{,p\vert_{\mathcal{S},p}}^*\Theta$ is a 
nonzero $1$-form on $\mathcal{S}_{,p}$: contradiction with the fact 
that $\phi_{,p}^*\omega$ vanishes on $\mathcal{S}_{,p}$.
\end{proof}

There is no statement if 
$\phi\in\mathrm{Bir}(\mathbb{C}^3)_{\mathrm{c}(\omega)}$ blows 
down $\mathcal{H}_\infty$ onto a point. Indeed
\[
\mathcal{K}\left(\frac{z_1}{z_2^2},\frac{z_1}{z_2^3}\right)=\left(
\frac{z_2+3z_0z_1}{z_2(z_2-2z_0z_1)},\frac{z_1}{z_2^2},\frac{z_1}{z_2^3}
\right)
\]
contracts $\mathcal{H}_\infty$ onto $\mathbf{e}_3\not\in\mathcal{H}_\infty$ 
but $\mathcal{K}(z_1z_2,z_1z_2^2)$ contracts $\mathcal{H}_\infty$ onto 
$\mathbf{e}_2\in\mathcal{H}_\infty$. But we get some result when 
$\phi\in\mathrm{Bir}(\mathbb{C}^3)_{\mathrm{c}(\omega)}$ blows 
down a surface distinct from $\mathcal{H}_\infty$ onto a point.

\begin{defi}
Let $\phi$ be a contact birational map of $\mathbb{P}^3$.
Let $\mathcal{S}=(f=0)$ be an irreducible 
surface blown down by $\phi$, and let $p$ be a smooth point of 
$\mathcal{S}$ such that $\phi$ and $V(\phi)$ are holomorphic at $p$. The 
multiplicity of contraction of $\phi$ at $p$ is the greatest integer 
$n$ such that $f_{,p}^n$ divides $V(\phi)$. One can check that~$n$ is 
independent on $p$. The integer $n$ is the \textbf{\textit{multiplicity 
of contraction of $\phi$ on $\mathcal{S}$}}.
\end{defi}

\begin{rem}
Let $\phi$ be a contact birational map of $\mathbb{P}^3$. 
If $\phi$ is holomorphic at 
$p\in\mathbb{P}^3\smallsetminus\mathcal{H}_\infty$, then $V(\phi)$ 
is too.
\end{rem}

\begin{eg}\label{eg:utile}
Let us consider the birational map $\phi$ defined in the affine chart 
$z_1=1$ by
\[
\phi=\left(\frac{z_0z_3^2}{(z_2+z_3)^2},\frac{z_2z_3}{(z_2+z_3)},z_3
\right);
\]
in this chart 
$\omega=\mathrm{d}z_2-\frac{z_0+z_2z_3}{z_3^2}\,\mathrm{d}z_3$ and 
one can check that $V(\phi)=\frac{z_3^2}{{z_2+z_3}^2}$. 
Furthermore $\mathcal{H}_\infty$ is blown down by $\phi$ onto the point 
$(0,0,0)$ ; the multiplicity of contraction of $\phi$ on 
$\mathcal{H}_\infty$ is thus $2$.
\end{eg}

\begin{pro}\label{pro:contrapoint}
Let $\phi$ be a map of 
$\mathrm{Bir}(\mathbb{C}^3)_{\mathrm{c}(\omega)}$ and let 
$\mathcal{S}$ be an irreducible surface distinct from~$\mathcal{H}_\infty$ 
blown down by $\phi$ onto a point $p$. If the multiplicity of contraction 
of $\phi$ on $\mathcal{S}$ is~$1$, then $p$ belongs to $\mathcal{H}_\infty$.
\end{pro}

\begin{rem}
As soon as the multiplicity of contraction of $\phi$ on $\mathcal{S}$ 
is $>1$, the point $p$ can be in 
$\mathbb{P}^3\smallsetminus\mathcal{H}_\infty$. Let us consider 
the map of $\mathrm{Bir}(\mathbb{C}^3)_{\mathrm{c}(\omega)}$ given 
in the affine chart $z_3=1$ by
\[
\left(\frac{z_2(nz_0z_1-z_2)}{z_2+(1-n)z_0z_1},z_1z_2^{n-1},z_1z_2^n\right)
\]
with $n\in\mathbb{Z}$.
The surface $z_2=0$ is blown down onto $\mathbf{e}_3\not\in\mathcal{H}_\infty$. 
One can check that $V(\phi)=\frac{z_1z_2^n}{z_2+(1-n)z_0z_1}$ 
so the multiplicity of contraction of $\phi$ on $z_2=0$ is $n$ if 
$n\geq 2$ and $0$ otherwise.
\end{rem}

\begin{proof}[Proof of Proposition \ref{pro:contrapoint}]
Assume by contradiction that $p=(p_0,p_1,p_2)$ does not belong to 
$\mathcal{H}_\infty$. Let $(f=0)$ be an equation of~$\mathcal{S}$; as the 
multiplicity of contraction of $\phi$ on $\mathcal{S}$ is $1$ one has 
$V(\phi)=fV_1$ with $V_{1\vert\mathcal{S}}$ generically regular. There exists 
a point $m\in\mathcal{S}$ such that $f_{,m}$ is a submersion and $\phi$ is 
holomorphic at $m$. One has 
$\phi_{,m}=(p_0+fA,p_1+fB,p_2+fC)$ with $A$, $B$, $C$ holomorphic and 
$\phi_{,m}^*\omega=V(\phi)\omega$ can be rewritten
\begin{equation}\label{eq:mat}
(fA+p_0)(f\mathrm{d}B+B\mathrm{d}f)+(f\mathrm{d}C+C\mathrm{d}f)
=fV_1(z_0\mathrm{d}z_1+\mathrm{d}z_2)
\end{equation}
This implies that there exists $C_1$ holomorphic such that $p_0B+C=fC_1$, 
{\it i.e.} $C=fC_1-p_0B$. Hence
\begin{equation}\label{eq:mat2}
(\ref{eq:mat})\Longleftrightarrow fA\mathrm{d}B+AB\mathrm{d}f+f
\mathrm{d}C_1+2C_1\mathrm{d}f=V_1(z_0\mathrm{d}z_1+\mathrm{d}z_2)
\end{equation}

The multiplicity of contraction of $\phi$ on $\mathcal{S}$ is $1$ 
hence $f$ does not divide $V_1$. Then $\mathcal{S}$ is invariant by 
$\omega$ and this gives a contradiction
 with the fact that $\mathcal{H}_\infty$ is the only invariant surface 
of~$\omega$.
\end{proof}

For elements in $\mathrm{Bir}(\mathbb{C}^3)_{\omega}$ we only 
have one statement that includes both cases of a surface contracted 
onto a point and onto a curve. Let us remark that in the case of a 
point, we don't need the assumption about the multiplicity of 
contraction; in the other one the statement shows that Proposition 
\ref{pro:contracourbe} applies to elements of 
$\mathrm{Bir}(\mathbb{C}^3)_{\mathrm{c}(\omega)}\smallsetminus\mathrm{Bir}(\mathbb{C}^3)_{\omega}$.

\begin{pro}\label{pro:subvar}
Let $\phi$ be a map of $\mathrm{Bir}(\mathbb{C}^3)_{\omega}$. 
If $\mathcal{S}$ is a surface distinct from $\mathcal{H}_\infty$ 
contracted by $\phi$, then $\phi(\mathcal{S})$ belongs to 
$\mathcal{H}_\infty$.
\end{pro}

\begin{proof}
From $\phi^*\omega=\omega$ one gets 
$\phi^*\big(\omega\wedge\mathrm{d}\omega\big)=\omega\wedge\mathrm{d}\omega=\mathrm{d}z_0\wedge\mathrm{d}z_1\wedge\mathrm{d}z_2$.
 Suppose that for $p\in\mathcal{S}$ generic $\phi(p)$ does not belong to 
$\mathcal{H}_\infty$. As $\mathrm{codim}\,\mathrm{Ind}\,\phi\geq 2$, 
the map $\phi$ is holomorphic at $p$. Since~$\phi$ preserves the 
volume form, $\phi$ is a diffeomorphism; hence $\phi$ cannot blow 
down a subvariety onto a curve or a point not contained in 
$\mathcal{H}_\infty$.
\end{proof}

\begin{eg}\label{egs:monomial}
If 
$\phi=(\phi_1,\phi_2)=\big(z_1^p z_2^q,z_1^r z_2^s\big)$, 
with $\left[\begin{array}{cc}
p & q \\
r & s
\end{array}\right]\in\mathrm{SL}(2;\mathbb{Z})$, then
\[
\mathcal{K}(\phi)=\left(z_1^{r-p}z_2^{s-q}\,
\frac{-r z_2+s z_0z_1}{p z_2-q z_0z_1},
z_1^p z_2^q,z_1^r z_2^s\right).
\]
Note that for any $\left[\begin{array}{cc}
p & q \\
r & s
\end{array}\right]\in\mathrm{SL}(2;\mathbb{Z})$ the map 
$\mathcal{K}(\phi)$ belongs to 
$\mathrm{Bir}(\mathbb{C}^3)_{\mathrm{c}(\omega)}\smallsetminus\mathrm{Bir}(\mathbb{C}^3)_\omega$.
\bigskip

For instance if $\left[\begin{array}{cc}
p & q \\
r & s
\end{array}\right]=\left[\begin{array}{cc}
-1 & 0 \\
0 & -1
\end{array}\right]$, \emph{i.e.} if 
$\sigma=\left(\frac{1}{z_0},\frac{1}{z_1}\right)$ is the Cremona 
involution, then
\[
\mathcal{K}(\sigma)=\mathcal{K}(\sigma^{-1})=\left(\frac{z_0z_1^2}{z_2^2},
\frac{1}{z_1},\frac{1}{z_2}\right)
\]
and 
$\mathrm{Ind}\,\mathcal{K}(\sigma)=\{z_0=z_2=0\}\cup\{z_0=z_3=0\}\cup\{z_1=z_2=0\}\cup\{z_1=z_3=0\}$; 
furthermore $z_2=0$ and $\mathcal{H}_\infty$ are blown down onto 
$\mathbf{e}_1$ and $z_1=0$ onto $\mathbf{e}_2$.
\end{eg}

\section{Some common properties}

\medskip

\subsection{Invariant curves and surfaces}

The following statement is a local statement of contact analytic 
geometry.

\begin{pro}\label{pro:per2}
Let $\phi$ be an element of $\mathrm{Aut}(\mathbb{C}^3)_\omega$ or 
$\mathrm{Bir}(\mathbb{C}^3)_\omega$. Suppose that $m$ is a 
periodic point of~$\phi$ and that there exists a germ of irreducible 
curve $\mathcal{C}$ invariant by $\phi$, passing through $m$. Then 
\smallskip
\begin{itemize}
\item either $\mathcal{C}$ is a curve of periodic points $($\emph{i.e.} 
$\phi^\ell_{\vert\mathcal{C}}=\mathrm{id}$ for some integer $\ell)$,
\smallskip
\item or $\mathcal{C}$ is a legendrian curve.
\end{itemize}
\end{pro}

Let us note that according to Proposition \ref{pro:per} we know that 
such a situation often occurs.

\begin{proof}
Assume that $\phi$ belongs to $\mathrm{Aut}(\mathbb{C}^3)_\omega$. Up 
to considering a well-chosen iterate of $\phi$ let us assume that~$m$ 
is a fixed point of $\phi$. Let $s\mapsto\gamma(s)$ be a local 
parametrization of $\mathcal{C}$ at $m$. Up to reparametrization one 
can suppose that $\gamma(0)=m$. Let $\varphi$ be the "restriction" 
to $\mathcal{C}$ of $\phi$, that is the local map 
$\varphi\colon \mathbb{C}_{,0}\to\mathbb{C}$ defined by $\varphi(0)=0$ and
\[
\forall\,s\in\mathbb{C}_{,0}\quad \phi(\gamma(s))=\gamma(\varphi(s)).
\]
On the one hand $\gamma^*\omega=\varepsilon(s)\mathrm{d}s$ and on 
the other hand 
$\gamma^*\omega=\gamma^*\phi^*\omega=(\phi\circ\gamma)^*\omega$ so
\[
\varepsilon(s)\mathrm{d}s=\varphi^*(\varepsilon(s)\mathrm{d}s)=
\varepsilon(\varphi)\varphi'\mathrm{d}s.
\]

Let us set 
$\widetilde{\varepsilon}(s)=\int_0^s\,\varepsilon(t)\mathrm{d}t$. 
One has 
$\left(\widetilde{\varepsilon}(\varphi)\right)'=\varepsilon(\varphi)\varphi'=\varepsilon(s)=\left(\widetilde{\varepsilon}(s)\right)'$
hence 
$\widetilde{\varepsilon}(\varphi)=\widetilde{\varepsilon}+\beta$ for 
some $\beta\in\mathbb{C}$. As $\varphi(0)=0$, one gets $\beta=0$ and 
$\widetilde{\varepsilon}(\varphi)=\widetilde{\varepsilon}$. Then:
\begin{itemize}
\smallskip
\item either $\widetilde{\varepsilon}=0$ therefore $\varepsilon=0$
and $\mathcal{C}$ is a legendrian curve.
\smallskip
\item or there exists some local coordinate for which 
$\widetilde{\varepsilon}=z^\ell$, $\varphi=e^{2\mathrm{i}\pi k/\ell}\,z$ 
and $\phi^\ell_{\vert\mathcal{C}}=\mathrm{id}$.
\end{itemize}
\end{proof}

If $\varphi$ is a polynomial automorphism of $\mathbb{C}^2$ that 
preserves a curve distinct from the line at infinity, then $\varphi$ 
is conjugate to a Jonqui\`eres polynomial automorphism 
(\cite{Brunella}); in particular $\varphi$ preserves a rational 
fibration. We have a similar statement in dimension $3$:

\begin{pro}\label{pro:surfinv}
If $\phi\in\mathrm{Aut}(\mathbb{C}^3)_{\omega}$ preserves a surface,
then
\[
\phi=\big(\varphi(z_0,z_1),z_2+b(z_0,z_1)\big)
\]
where $\varphi$ is $\mathrm{Aut}(\mathbb{C}^2)$-conjugate to a 
Jonqui\`eres polynomial automorphism.
\end{pro}

\begin{proof}
Let us write $\phi$ as 
$\big(\phi_0(z_0,z_1),\phi_1(z_0,z_1),z_2+b(z_0,z_1)\big)$ and set 
$\varphi=(\phi_0,\phi_1)$.

First note that if $b\equiv 0$ then 
$\phi_0\mathrm{d}\phi_1-z_0\mathrm{d}z_1=0$; as a result 
$\phi_1=\phi_1(z_1)$ and $\varphi$ is a Jonqui\`eres polynomial 
automorphism.

Let us now assume that the surface $\mathcal{S}$ preserved 
by $\phi$ is described by
\[
a_\ell(z_0,z_1)z_2^\ell+a_{\ell-1}(z_0,z_1)z_2^{\ell-1}+a_{\ell-2}(z_0,
z_1)z_2^{\ell-2}+\ldots=0
\]
where $a_i\in\mathbb{C}[z_0,z_1]$, or equivalently by
\[
z_2^\ell+\widetilde{a}_{\ell-1}(z_0,z_1)z_2^{\ell-1}+
\widetilde{a}_{\ell-2}(z_0,z_1)z_2^{\ell-2}+\ldots=0
\]
where $\widetilde{a}_{i}=a_i/a_\ell$. Writing that $\mathcal{S}$
 is invariant by $\phi$ one gets that
\begin{eqnarray*}
& & \big(z_2+b(z_0,z_1)\big)^\ell+\widetilde{a}_{\ell-1}\big(\varphi(
z_0,z_1)\big)\big(z_2+b(z_0,z_1)\big)^{\ell-1}+\widetilde{a}_{\ell-2}
\big(\varphi(z_0,z_1)\big)\big(z_2+b(z_0,z_1)\big)^{\ell-2}+\ldots\\
& & \hspace{1cm}=z_2^\ell+\widetilde{a}_{\ell-1}(z_0,z_1)z_2^{\ell-1}+
\widetilde{a}_{\ell-2}(z_0,z_1)z_2^{\ell-2}+\ldots
\end{eqnarray*}
Looking at terms in $z_2^{\ell-1}$ one gets that 
$\ell b(z_0,z_1)=\widetilde{a}_{\ell-1}(z_0,z_1)-\widetilde{a}_{\ell-1}\big(\varphi(z_0,z_1)\big)$.

\begin{itemize}
\smallskip
\item If $\widetilde{a}_{\ell-1}$ is constant, then $b\equiv 0$ 
and as we just see $\varphi$ is a Jonqui\`eres polynomial 
automorphism.
\smallskip
\item Otherwise $\phi$ is conjugate (in 
$\mathrm{Bir}(\mathbb{P}^3)$) via $\left(z_0,z_1,z_2+\frac{\widetilde{a}_{\ell-1}}{\ell}\right)$ 
to $\psi=(\varphi,z_2)$. The map $\psi$ preserves 
$\widetilde{\omega}=z_0\mathrm{d}z_1+\mathrm{d}\left(z_2+\frac{\widetilde{a}_{\ell-1}}{\ell}\right)$, 
the surface $\widetilde{\mathcal{S}}$ given by
\[
z_2^\ell+\widetilde{a}_{\ell-2}(z_0,z_1)z_2^{\ell-2}+\widetilde{a}_{\ell-3}
(z_0,z_1)z_2^{\ell-3}+\ldots=0
\]
\end{itemize}
\smallskip
and thus $\widetilde{a}_i(\varphi)=\widetilde{a}_i$. If one of the 
$\widetilde{a}_i$ is non-constant, then $\varphi$ is a 
Jonqui\`eres polynomial automorphism. 
Otherwise $\widetilde{\mathcal{S}}=\displaystyle\cup_j(z_2=c_j)$; 
up to take an iterate $\psi^k$ of $\psi$
 one can suppose that any $z_2=c_j$ is invariant. Consider $z_2=c_0$;
 up to a well-chosen translation (that belongs to 
$\mathrm{Bir}(\mathbb{C}^3)_\omega$) the hypersurface $z_2=0$ 
is invariant, that is $\psi^k$ is a Jonqui\`eres 
map and so does $\psi$.
\end{proof}

\begin{eg}
For any $n\geq 1$ consider 
$\phi=\left(z_0+z_1^n,z_1,z_2-\frac{z_1^{n+1}}{n+1}\right)$ in 
$\mathrm{Aut}(\mathbb{C}^3)_{\omega}$. The map $\varphi=(z_0+z_1^n,z_1)$ 
is a Jonqui\`eres polynomial automorphism. The surface 
$\mathcal{S}$ given by 
$z_2+\frac{z_0z_1}{n+1}=0$, is invariant by $\phi$. The foliation 
induced by $\omega$ on $\mathcal{S}$ is described by 
the linear differential equation $nz_0\mathrm{d}z_1-z_1\mathrm{d}z_0$.
In fact the functions $z_2+\frac{z_0z_1}{n+1}$ and $z_1$ are invariant 
by $\phi$ and the commutative Lie algebra generated by the vector
fields 
$\frac{\partial}{\partial z_0}+\frac{z_1}{n+1}\frac{\partial}{\partial z_2}$ 
and $\frac{\partial}{\partial z_2}$ are invariant by $\phi$.

In general an element of $\mathrm{Aut}(\mathbb{C}^3)_\omega$ has
no invariant surface. For instance there is no polynomial solution to
\[
-a\big(\varphi(z_0,z_1)\big)+a(z_0,z_1)=-\frac{z_1^{n+1}}{n+1}+\beta
\]
with $\varphi=(z_0+z_1^n,z_1)$ as soon as $\beta\not=0$.
\end{eg}

\begin{rem}\label{z2omega}
If $\phi\in\mathrm{Bir}(\mathbb{C}^3)_\omega$ preserves $z_2=0$,
then $\phi$ belongs to the Klein family; more precisely 
$\phi=\left(\frac{z_0}{\nu'(z_1)},\nu(z_1),z_2\right)$ with 
$\nu\in\mathrm{PGL}(2;\mathbb{C}(z_1))$. Indeed since $\phi$ belongs
to $\mathrm{Bir}(\mathbb{C}^3)_\omega$, 
\[
\phi=\big(\phi_0(z_0,z_1),\phi_1(z_0,z_1),z_2+b(z_0,z_1)\big).
\]
But $\phi$ preserves $z_2=0$ so $b\equiv 0$ and $\phi^*\omega=\omega$
implies that $\phi_1=\nu(z_1)$ with 
$\nu\in\mathrm{PGL}(2;\mathbb{C}(z_1))$ and $\phi_0=\frac{z_0}{\nu'(z_1)}$.

Of course there are more general contact maps that preserve $z_2=0$; 
let us give some examples:
\[
\mathcal{K}\left(z_1,\frac{z_2}{a(z_1)z_2+1}\right),
\qquad
\mathcal{K}\big(z_1+P(z_2),z_2\big)
\]
where $a\in\mathbb{C}(z_1)^*$ and $P\in\mathbb{C}[z_1]$.
\end{rem}

Let $\phi$ be an element of 
$\mathrm{Bir}(\mathbb{C}^3)_{\omega}$. Suppose that 
$\phi$ preserves a surface $\mathcal{S}$ distinct from 
$\mathcal{H}_\infty$. The contact form is non-zero on $\mathcal{S}$ 
so induces a foliation $\mathcal{F}$ on $\mathcal{S}$, necessarily 
invariant by $\phi$; let us describe 
$(\mathcal{S},\phi_{\vert\mathcal{S}},\mathcal{F})$:

\begin{pro}
Let $\phi$ be an element of 
$\mathrm{Bir}(\mathbb{C}^3)_\omega$ that preserves a surface 
distinct from $\mathcal{H}_\infty$. Then $\phi$ 
is~$\mathrm{Bir}(\mathbb{P}^3)$-conjugate to 
$(\varphi(z_0,z_1),z_2)$ with $\varphi$ in 
$\mathrm{Bir}(\mathbb{P}^2)$. The map $\varphi$ preserves 
a codimension $1$ foliation given by a closed $1$-form. As a 
consequence $\phi$ preserves a "vertical" foliation and a rational
function~$z_2+~a(z_0,z_1)$.
\end{pro}

\begin{proof}
Let us denote by $\mathcal{S}$ the surface invariant by 
$\phi=\big(\varphi(z_0,z_1),z_2+b(z_0,z_1)\big)$ with 
$\varphi\in\mathrm{Bir}(\mathbb{P}^2)$. One can assume 
that $\mathcal{S}$ is given by
\[
z_2^\ell+a_{\ell-1}(z_0,z_1)z_2^{\ell-1}+\ldots=0
\]
The fact that $\mathcal{S}$ is invariant by $\phi$ implies that 
$a_{\ell-1}(z_0,z_1)-a_{\ell-1}\big(\varphi(z_0,z_1)\big)=\ell\, b(z_0,z_1)$. 
Let us consider the map 
$\psi=\left(z_0,z_1,z_2+\frac{a_{\ell-1}(z_0,z_1)}{\ell}\right)$. One has
\[
\widetilde{\phi}=\psi\phi\psi^{-1}=\left(\varphi(z_0,z_1),z_2+b(z_0,z_1)
-\frac{a_{\ell-1}(z_0,z_1)}{\ell}+\frac{a_{\ell-1}\big(\varphi(z_0,z_1)
\big)}{\ell}\right)=\big(\varphi(z_0,z_1),z_2\big)
\]
As $\mathcal{S}$ and $\omega$ are invariant by $\phi$, the restriction 
$\phi_{\vert\mathcal{S}}$ preserves the foliation induced by $\omega$ on 
$\mathcal{S}$, and $\widetilde{\phi}$ preserves the "vertical" 
foliation given by $z_0\mathrm{d}z_1-\mathrm{d}a_{\ell-1}(z_0,z_1)$. 
Therefore $\varphi$ preserves a codimension $1$ foliation given by 
a closed $1$-form.
\end{proof}

\begin{eg}
If $\phi=(z_2,z_1z_2^n)$, then 
$\mathcal{K}(\phi)=\left(-\frac{z_2^n}{z_0}+nz_1,z_1z_2^n,z_2\right)$ 
belongs to 
$\mathrm{Bir}(\mathbb{C}^3)_{\mathrm{c}(\omega)}\smallsetminus\mathrm{Bir}(\mathbb{C}^3)_{\omega}$
preserves the surface $z_1=0$ and also $z_2=$ cst.
\end{eg}

\medskip

\subsection{Dynamical properties}

\smallskip

Let us first focus on periodic points.

Let $\phi$ be a birational map of $\mathbb{P}^n$; a point 
$p$ is a \textbf{\textit{periodic point}} of $\phi$ of period $\ell$
 if $\phi$ is holomorphic on a neighborhood of any point of 
$\{\phi^j(q)\,\vert\, j=0,\,\ldots,\ell-1\}$ and if $\phi^\ell(q)=q$ 
and $\phi^j(q)\not=q$ for $1\leq j\leq \ell-1$.

Recall that a polynomial automorphism of $\mathbb{C}^2$ of 
H\'enon type (\emph{see} \cite{FriedlandMilnor}) has an 
infinite number of hyperbolic periodic points. For any of these points 
$p$ of period $\ell_p$ there exists a stable manifold $\mathrm{W}^s(p)$ 
defined as the set of points that move towards the orbit of $p$ by 
positive iteration of~$\varphi^{\ell_p}$; such a $\mathrm{W}^s(p)$ is 
an immersion from $\mathbb{C}$ to $\mathbb{C}^2$. 
Remark that even if $\mathrm{W}^s(m)\not=\mathrm{W}^s(p)$ are 
different as soon as $p$ and~$m$ have distinct orbits one has 
$\overline{\mathrm{W}^s(m)}=\overline{\mathrm{W}^s(p)}$. The 
Julia set of $\varphi$ is the topological boundary of the 
set of points with bounded positive orbits. One can prove that the 
Julia set of $\varphi$ is equal to the closure of any of 
the stable manifold. Hence its topology is very complicated: this 
set contains an infinite number of immersions of $\mathbb{C}$ and
pairwise distinct (\cite{FriedlandMilnor}).

\begin{eg}
Let us consider a polynomial automorphism $\varphi$ of H\'enon 
type given by $\varphi=(\beta z_1+z_0^2,-\gamma z_0)$. A $\varsigma$-lift 
of $\varphi$ to $\mathrm{Aut}(\mathbb{C}^3)_{\mathrm{c}(\omega)}$ is
\[
\phi=\Big(\beta z_1+z_0^2,-\gamma z_0,\gamma\beta z_2+\gamma\beta 
z_0z_1+\frac{\gamma}{3}z_0^3\Big)
\]
Take a periodic point $(p_0,p_1)$ of $\varphi$ of period $k$; 
then as 
$\phi^k=\big(\varphi^k(z_0,z_1),(\gamma\beta)^kz_2+f(z_0,z_1)\big)$ 
one gets, as soon as $\gamma\beta$ is not a root of unity, that 
there exists $p_2$ such that $\phi^k(p_0,p_1,p_2)=(p_0,p_1,p_2)$.
\end{eg}

More generally, one can state: 

\begin{pro}
Let $\phi$ the element of 
$\mathrm{Bir}(\mathbb{C}^3)_{\mathrm{c}(\omega)}$ of the following 
type
\[
\phi=\big(\varphi,\det\mathrm{jac}\varphi z_2+b(z_0,z_1)\big)
\]
with $\varphi$ in $\mathrm{Bir}(\mathbb{P}^2)$ and $b$ in 
$\mathbb{C}(z_0,z_1)$.

If $\det\mathrm{jac}\varphi$ is not a root of unity, then any 
periodic point of $\varphi$ can be lifted into a periodic point 
of $\phi$.
\end{pro}

\begin{cor}\label{cor:henonperiodic}
Let $\varphi$ be a polynomial automorphism of $\mathbb{C}^2$ of 
H\'enon type. A $\varsigma$-lift of $\varphi$ has an infinite 
number of periodic points that lift the hyperbolic periodic points 
of $\varphi$.
\end{cor}

\begin{question}\label{quest:question}
Let $\varphi$ be a H\'enon automorphism and let $\phi$
be a $\varsigma$-lift of $\varphi$.
The  closure of the hyperbolic periodic points of $\varphi$ 
is the Julia set of $\varphi$; in particular it is
 a Cantor set. Is the closure of the set of periodic points of 
$\phi$ a Cantor set~?
\end{question}

Let us consider a H\'enon automorphism 
$\varphi=(\varphi_1,\varphi_2)$ and let $m$ be an hyperbolic 
periodic point of $\varphi$; then the matrix
\[
\left[
\begin{array}{cc}
-\frac{\partial\varphi_2}{\partial z_1} &\frac{\partial\varphi_2}
{\partial z_2}\\
\frac{\partial\varphi_1}{\partial z_1} & -\frac{\partial 
\varphi_1}{\partial z_2}
\end{array}
\right]
\]
is a non-parabolic one and so 
$z_0\mapsto \frac{-\frac{\partial\varphi_2}{\partial z_1}+\frac{\partial\varphi_2}{\partial z_2}\,z_0}{\frac{\partial\varphi_1}{\partial z_1}-\frac{\partial \varphi_1}{\partial z_2}\,z_0}$ 
has two fixed points. We can thus state the follo\-wing:

\begin{pro}\label{pro:per}
Let $\varphi$ be an automorphism of $\mathbb{C}^2$ of H\'enon 
type; to any periodic point of period~$\ell$ of $\varphi$ 
corresponds two periodic points of period $\ell$ of 
$\mathcal{K}(\varphi)\in\mathrm{Bir}(\mathbb{C}^3)_{\mathrm{c}(\omega)}$. 
\end{pro}

A similar question as Question \ref{quest:question} is the 
following:

\begin{question}
Let $\varphi$ be a polynomial automorphism of $\mathbb{C}^2$ of 
H\'enon type; what is the topology of the distribution 
of periodic points of $\mathcal{K}(\varphi)$ ? Is it a discrete 
set ? Is its closure a Cantor set ?
\end{question}

\begin{rem}
Let us consider an element 
$\big(\phi_0(z_0,z_1),\phi_1(z_0,z_1),z_2+b(z_0,z_1)\big)$ of 
$\mathrm{Bir}(\mathbb{C}^3)_\omega$. Then 
$\phi_t=\big(\phi_0(z_0,z_1),\phi_1(z_0,z_1),z_2+b(z_0,z_1)+t\big)$ 
belongs to $\mathrm{Bir}(\mathbb{C}^3)_\omega$. If 
$p=(p_0,p_1,p_2)$ is a fixed point of $\phi_t$, then $(p_0,p_1)$ is 
a fixed point of $\varphi=(\phi_0,\phi_1)$ and $b(p_0,p_1)+t=0$. In
 particular if $\varphi$ only has isolated fixed points (that is 
$\varphi$ has no curve of fixed points, which is the case in 
general), then $\phi_t$ has no fixed points for $t$ generic.

Similarly, if $\varphi$ has a countable number of periodic points, 
then for $t$ generic $\phi_t$ has no periodic points.
\end{rem}

\smallskip

We will look at degree and degree growths of some contact 
birational maps.

In the $2$-dimensional case, that is if $\varphi$ belongs 
to $\mathrm{Aut}(\mathbb{C}^2)$, or $\mathrm{Bir}(\mathbb{P}^2)$, 
then $\deg\varphi=\deg\varphi^{-1}$. This equality is not true in
higher dimension; for instance if 
\[
\phi=\big(z_0^2+z_2^2+z_1,z_2^2+z_0,z_2\big),
\] 
then $\phi^{-1}=\big(z_1-z_2^2,z_0-(z_1-z_2^2)^2-z_2^2,z_2\big)$).
What happens in our context ? The equality $\deg\varphi=\deg\varphi^{-1}$
still does not hold; indeed if $(\phi_0,\phi_1,z_2+b(z_0,z_1))$ 
belongs to $\mathrm{Aut}(\mathbb{C}^3)_\omega$, then 
$-\mathrm{d}b=\phi_0\mathrm{d}\phi_1-z_0\mathrm{d}z_1$ and
$\deg b=\deg\phi_0+\deg \phi_1$. For instance if 
$\varphi=\big(z_0+(z_1^3-z_0)^2,z_1^3-z_0\big)$, then 
\[
\varphi^{-1}=\big((z_0-z_1^2)^3-z_1,z_0-z_1^2\big).
\]
Hence the degree of the $\varsigma$-lifts of $\varphi$ (resp. 
$\varphi^{-1}$) is $9$ (resp. $8$).
\medskip

Let $\phi$ and $\psi$ be two birational self-maps of 
$\mathbb{P}^3$. We will say that \textbf{\textit{the 
degree growths of $\phi$ and $\psi$ are of the same order}} 
if one of the following holds
\smallskip
\begin{itemize}
\item $(\deg\phi^n)_n$ and $(\deg\psi^n)_n$ are bounded,
 \smallskip
\item there exist an integer $k$ such that 
$\displaystyle\lim_{n\to +\infty}\frac{\deg\phi^n}{n^k}$ and 
$\displaystyle\lim_{n\to +\infty}\frac{\deg\psi^n}{n^k}$ are 
finite and nonzero,
\smallskip
\item $(\deg\phi^n)_n$ and $(\deg\psi^n)_n$ grow exponentially.
\end{itemize} 
\smallskip

Let $\varphi$ be a polynomial automorphism of $\mathbb{C}^2$; 
let us recall that $\varphi$ has either a bounded growth or 
an exponential one (\cite{FriedlandMilnor}). Denote by $\phi$ 
a $\varsigma$-lift of $\varphi$ to 
$\mathrm{Aut}(\mathbb{C}^3)_{\mathrm{c}(\omega)}$
\[
\phi=\big(\varphi,\det\,\mathrm{jac}\,\varphi\, z_2+b(z_0,z_1)
\big)
\]
Note that $b$ belongs to $\mathbb{C}[z_0,z_1]$ and so 
$\deg b(\varphi^j(z_0,z_1))\leq\deg b\deg \varphi^j$ for any $j$. 
Hence 
\[
\deg\varphi^n\leq\deg \phi^n\leq\max(\deg \varphi^n,\deg b\deg \varphi^{n-1})
\]
and 
\smallskip
\begin{itemize}
\item if $(\deg \varphi^n)_n$ is bounded, then $(\deg \phi^n)_n$ 
is bounded,
\smallskip
\item if $(\deg \varphi^n)_n$ grows exponentially, then 
$(\deg \phi^n)_n$ grows exponentially.
\end{itemize}
\smallskip
Remark that if $\psi$ is a polynomial automorphism of 
$\mathbb{C}^3$ linear growth is also possible 
(\cite{BonifantFornaess}) and this even\-tuality does not appear
when we look at elements of 
$\mathrm{Aut}(\mathbb{C}^3)_{\mathrm{c}(\omega)}$.

In the case of the $\varsigma$-lift of an exact element of 
$\mathrm{Bir}(\mathbb{C}^2)_\eta$ we cannot give formula 
because we are not dealing with polynomials. But the degree 
growth of a $\varsigma$-lift $\phi$ of an exact element $\varphi$ 
of $\mathrm{Bir}(\mathbb{C}^2)_\eta$ and the degree 
growth of $\varphi$ are the same. Indeed set 
$\varphi^n=(\varphi_{0,n},\varphi_{1,n})$ for any $n\geq 1$. On 
the one hand
\[
\phi^n=\big(\varphi_{0,n},\varphi_{1,n},z_2+b(z_0,z_1)+b(
\varphi_{0,1},\varphi_{1,1})+b(\varphi_{0,2},\varphi_{1,2})+\ldots
+b(\varphi_{0,n-1},\varphi_{1,n-1})\big)
\]
with $\mathrm{d}b=z_0\mathrm{d}z_1-\varphi_0\mathrm{d}\varphi_1$, 
but on the other hand 
$\phi^n=\big(\varphi_{0,n},\varphi_{1,n},z_2+\widetilde{b}(z_0,z_1)\big)$ 
with 
$\mathrm{d}\widetilde{b}=z_0\mathrm{d}z_1-\varphi_{0,n}\mathrm{d}\varphi_{1,n}$. 
Using this last writing one gets the statement.

\medskip

Let $\phi$ be a birational self-map of $\mathbb{P}^2$. For any 
$n\geq 1$ set 
$\phi^n=(\phi_{1,n},\phi_{2,n})=\Big(\frac{P_{1,n}}{Q_{1,n}},\frac{P_{2,n}}{Q_{2,n}}\Big)$ 
with $P_{i,n}$, $Q_{i,n}\in\mathbb{C}[z_0,z_1]$ without common factor; 
denote by $p_{i,q}$ (resp. $q_{i,n}$) the degree of $P_{i,n}$ (resp. 
$Q_{i,n}$). Of course 
$\deg\phi^n=\max(p_{1,n}+q_{2,n},p_{2,n}+q_{1,n},q_{1,n}+q_{2,n})$ and since 
\[
\mathcal{K}(\phi)^n=\mathcal{K}(\phi^n)=\left(\frac{Q_{2,n}^2}{Q_{1,n}^2}\,
 \frac{P_{2,n}\frac{\partial Q_{2,n}}{\partial z_1}-Q_{2,n}\frac{\partial 
P_{2,n}}{\partial z_1}+\left(Q_{2,n}\frac{\partial P_{2,n}}{\partial z_2}-
P_{2,n}\frac{\partial Q_{2,n}}{\partial z_2}\right)z_0}{Q_{1,n}\frac{\partial
 P_{1,n}}{\partial z_1}-P_{1,n}\frac{\partial Q_{1,n}}{\partial z_1}-\left(
Q_{1,n}\frac{\partial P_{1,n}}{\partial z_2}-P_{1,n}\frac{\partial Q_{1,n}}
{\partial z_2}\right)z_0}
,\frac{P_{1,n}}{Q_{1,n}},\frac{P_{2,n}}{Q_{2,n}}\right)
\]
one gets 
$\deg\phi^n\leq\deg\mathcal{K}(\phi)^n\leq\max(4q_{2,n}+p_{2,n}+1,2p_{1,n}+2q_{1,n}+q_{2,n}+1,p_{2,n}+3q_{1,n}+p_{1,n}+1)$.

\begin{pro}
\begin{itemize}
\item Assume that $\mathrm{G}=\mathrm{Aut}(\mathbb{C}^2)$ or 
$\mathrm{G}=\mathrm{Bir}(\mathbb{C}^2)_\eta$. Let $\varphi$ 
be an element of $\mathrm{G}$, and let~$\phi$ be a $\varsigma$-lift 
of $\varphi$. The degree growths of $\varphi$ and $\phi$ are of the 
same order.
\smallskip
\item Let $\varphi$ be a birational self-map of the complex projective 
plane, and let us consider $\mathcal{K}(\varphi)$ the image of $\varphi$ 
by $\mathcal{K}$. The degree growths of $\varphi$ and 
$\mathcal{K}(\varphi)$ are of the same order.
\end{itemize}
\end{pro}

\smallskip

Let us end this section by some considerations about centralisers of 
contact birational maps. 

If $\mathrm{G}$ is a group and $f$ an element of $\mathrm{G}$, we denote 
by $\mathrm{Cent}(f,\mathrm{G})$ the centraliser of $f$ in~$\mathrm{G}$, 
that is
\[
\mathrm{Cent}(f,\mathrm{G})=\big\{g\in\mathrm{G}\,\vert\, fg=gf\big\}
\]

Let $\varphi$ be a polynomial automorphism of $\mathbb{C}^2$, then 
(\cite{FriedlandMilnor, Lamy})
\smallskip
\begin{itemize}
\item either $\varphi$ is conjugate to an element of $\mathrm{J}_2$ and 
$\mathrm{Cent}(\varphi,\mathrm{Aut}(\mathbb{C}^2))$ is uncountable;
\smallskip
\item or $\varphi$ is of H\'enon type and the centraliser of 
$\varphi$ is isomorphic to $\mathbb{Z}\rtimes\mathbb{Z}/p\mathbb{Z}$
for some $p$.
\end{itemize}
\smallskip
Let $\mathcal{H}$ be the set of polynomial automorphisms of $\mathbb{C}^2$
of H\'enon type. 

\begin{pro}
Let $\varphi$ be a polynomial automorphism of $\mathbb{C}^2$ and 
let $\phi$ be one of its $\varsigma$-lift. 
\smallskip
\begin{itemize}
\item If $\det\mathrm{jac}\,\varphi=1$, then 
$\mathrm{Cent}(\phi,\mathrm{Aut}(\mathbb{C}^3)_{\omega})$ is uncountable
and isomorphic to $\mathrm{Cent}(\phi)\rtimes\mathbb{C}$.
\smallskip
\item If $\det\mathrm{jac}\varphi\,\not=1$ and $\varphi$ belongs to 
$\mathcal{H}$, then 
$\mathrm{Cent}(\phi,\mathrm{Aut}(\mathbb{C}^3)_{\mathrm{c}(\omega)})$ is 
countable and isomorphic to $\mathrm{Cent}(\varphi)$.
\end{itemize}
\end{pro}

\begin{proof}
One can look at the restriction of $\varsigma$ to 
$\mathrm{Cent}(\phi,\mathrm{Aut}(\mathbb{C}^3)_{\mathrm{c}(\omega)})$:
\[
\varsigma_{\vert\mathrm{Cent}(\phi,\mathrm{Aut}(\mathbb{C}^3)_{\mathrm{c}(\omega)})}\colon
\mathrm{Cent}(\phi,\mathrm{Aut}(\mathbb{C}^3)_{\mathrm{c}(\omega)})\to
\mathrm{Cent}(\varphi,\mathrm{Aut}(\mathbb{C}^2))
\]
Of course 
\[
\ker\varsigma_{\vert\mathrm{Cent}(\phi,\mathrm{Aut}(\mathbb{C}^3)_{\mathrm{c}(\omega)})}
\subset\big\{(z_0,z_1,z_2+\beta)\,\vert\,\beta\in\mathbb{C}\big\}.
\]

If $\det\mathrm{jac}\,\varphi=1$, {\it i.e.} $\varphi$ belongs to 
$\mathrm{Aut}(\mathbb{C}^2)_\eta$, then 
\[
\ker\varsigma_{\vert\mathrm{Cent}(\phi,\mathrm{Aut}(\mathbb{C}^3)_{\mathrm{c}(\omega)})}
=\big\{(z_0,z_1,z_2+\beta)\,\vert\,\beta\in\mathbb{C}\big\}
\]
and the centraliser of a $\varsigma$-lift of $\varphi$ is always 
uncountable even if $\mathrm{Cent}(\varphi,\mathrm{Aut}(\mathbb{C}^2))$ is 
countable. 

If $\det\mathrm{jac}\,\varphi\not=1$, {\it i.e.} $\varphi$ belongs to 
$\mathrm{Aut}(\mathbb{C}^2)\smallsetminus\mathrm{Aut}(\mathbb{C}^2)_\eta$, 
then 
$\ker\varsigma_{\vert\mathrm{Cent}(\phi,\mathrm{Aut}(\mathbb{C}^3)_{\mathrm{c}(\omega)})}=\{\mathrm{id}\}$ 
and 
\[
\mathrm{Cent}(\phi,\mathrm{Aut}(\mathbb{C}^3)_{\mathrm{c}(\omega)})
\hookrightarrow\mathrm{Cent}(\varphi,\mathrm{Aut}(\mathbb{C}^2))
\]
In particular if $\varphi$ belongs to 
$(\mathrm{Aut}(\mathbb{C}^2)\smallsetminus\mathrm{Aut}(\mathbb{C}^2)_\eta)\cap\mathcal{H}$, 
then $\mathrm{Cent}(\phi,\mathrm{Aut}(\mathbb{C}^3)_{\mathrm{c}(\omega)})$ is 
countable.
\end{proof}

\begin{rem}
Contrary to the $2$-dimensional case there exist some $\varphi$ in 
$\mathrm{Aut}(\mathbb{C}^3)_\omega$ such that
\smallskip
\begin{itemize}
\item $\mathrm{Cent}(\phi,\mathrm{Aut}(\mathbb{C}^3)_\omega)$ is uncountable,
\smallskip
\item and $(\deg\phi^n)_{n\in\mathbb{N}}$ grows exponentially.
\end{itemize}
\end{rem}

A similar reasoning leads to:

\begin{pro}
Let $\varphi\in\mathrm{Bir}(\mathbb{C}^2)_\eta$ be an exact map, 
and let $\phi$ be one of its $\varsigma$-lifts. Then 
$\mathrm{Cent}(\phi,\mathrm{Bir}(\mathbb{C}^3)_\omega)$ is uncountable.
\end{pro}

Let $\mathrm{G}=\mathrm{Aut}(\mathbb{C}^2)$ or 
$\mathrm{G}=\mathrm{Bir}(\mathbb{C}^2)_\eta$. Let $\varphi$ be an
 element of $\mathrm{G}$, and let $\phi$ be one of its $\varsigma$-lift. 
In the following examples we look at the links between 
the $\varsigma$-lift of $\mathrm{Cent}(\varphi,\mathrm{G})$ and 
$\mathrm{Cent}(\phi,\mathrm{G}')$ 
where $\mathrm{G}'=\mathrm{Aut}(\mathbb{C}^3)_{\mathrm{c}(\omega)}$ 
or~$\mathrm{Bir}(\mathbb{C}^3)_{\mathrm{c}(\omega)}$.

\begin{eg}
In this example we give a polynomial automorphism $\varphi$ and maps 
in $\mathrm{Cent}\big(\varphi,\mathrm{Aut}(\mathbb{C}^2)\big)$
whose only one $\varsigma$-lift belongs to 
$\mathrm{Aut}(\phi,\mathrm{Aut}(\mathbb{C}^3)_{\mathrm{c}(\omega)})$
where $\phi$ denotes a $\varsigma$-lift of $\varphi$. 

\smallskip

Let us now consider the H\'enon automorphism $\varphi$ given by
\[
\varphi=(\delta z_1,\beta z_1^k-\gamma z_0)
\] 
where $\delta$, $\beta$, $\gamma$ are complex numbers such that 
$\delta\beta\not=0$, $\delta\beta\not=1$ and $k\geq 4$. The map 
\[
\phi=\Big(\delta z_1,\beta z_1^k-\gamma z_0,\delta\gamma z_2+\delta\gamma
 z_0z_1-\frac{\delta\beta}{k+1}z_1^{k+1}\Big)
\]
is a $\varsigma$-lift of $\varphi$. One can check that $(\zeta z_0,\zeta z_1)$, 
where $\zeta\in\mathbb{C}^*$ such that $\zeta^k=\zeta$, commutes with $\varphi$. 
Among the $\varsigma$-lifts $(\zeta z_0,\zeta z_1,\zeta^2z_2+\beta)$, 
$\beta\in\mathbb{C}$, only one commutes with $\phi$.  
\end{eg}

\begin{eg} 
We consider a polynomial automorphism $\varphi$, a subgroup $\mathrm{G}$ of
$\mathrm{Cent}(\varphi,\mathrm{Aut}(\mathbb{C}^2))$ and $\mathrm{G}_\varsigma$
its~$\varsigma$-lift. In the first example the inclusion 
$\mathrm{G}_\varsigma\subset\mathrm{Cent}(\phi,\mathrm{Aut}(\mathbb{C}^3)_{\mathrm{c}(\omega)})$
holds whereas in the second example it doesn't.

\medskip

Let us consider the polynomial automorphism 
$\varphi=(\beta^dz_0+\beta^dz_1^dQ(z_1^r),\beta z_1)$ with $\beta\in\mathbb{C}^*$, 
$Q\in\mathbb{C}[z_1]$ and $d$, $r\in\mathbb{N}$. One can check that
\[
\mathrm{G}=\big\{(z_0+\gamma z_1^d,z_1)\,\vert\,\gamma\in\mathbb{C}\big\}\subset
\mathrm{Cent}(\varphi,\mathrm{Aut}(\mathbb{C}^2))
\]
The map 
$\phi=\big(\beta^dz_0+\beta^dz_1^dQ(z_1^r),\beta z_1,\beta^{d+1}z_2-\beta P(z_1)\big)$ 
with $P'(z_1)=z_1^qQ(z_1^r)$ is a $\varsigma$-lift of $\varphi$. Let 
$\mathrm{G}_\varsigma$ be the $\varsigma$-lift of $\mathrm{G}$; the group
\[
\mathrm{G}_\varsigma=\left\{\left(z_0+\gamma z_1^d,z_1,z_2-\frac{\gamma z_1^{d+1}}{d+1}
\right)\,\vert\,\gamma\in\mathbb{C}\right\}
\]
is here contained in 
$\mathrm{Cent}(\phi,\mathrm{Aut}(\mathbb{C}^3)_{\mathrm{c}(\omega)})$.

\medskip

Let $\varphi$ be the polynomial automorphism given by 
$\varphi=(z_0+z_1^2,\lambda z_1)$ with $\lambda\in\mathbb{C}^*$ and $\lambda^2\not=1$. 
A $\varsigma$-lift of $\varphi$ to $\mathrm{Aut}(\mathbb{C}^3)_{\mathrm{c}(\omega)}$ 
is 
\[
\phi=\Big(z_0+z_1^2,\lambda z_1,\lambda z_2-\frac{z_1^3}{3}+\mu\Big)
\]
for some $\mu\in\mathbb{C}$. Note that 
\[
\mathrm{G}=\Big\{\Big(\delta z_0+\frac{\gamma^2-\delta}{\lambda^2-1}\,z_1+
\varepsilon,\gamma z_1\Big)\,\vert\,\delta,\,\gamma\in\mathbb{C}^*,\,
\varepsilon\in\mathbb{C}\Big\}
\]
is contained in $\mathrm{Cent}(\varphi,\mathrm{Aut}(\mathbb{C}^2))$. Let us denote 
by $\mathrm{G}_\varsigma$ the $\varsigma$-lift of $\mathrm{G}$; a direct computation 
shows that 
\[
\mathrm{G}_\varsigma=\Big\{\Big(\delta z_0+\frac{\gamma^2-\delta}{\lambda^2-1}\,z_1+
\varepsilon,\gamma z_1,\delta\gamma z_2-\frac{\gamma(\gamma^2-\delta)}{3
(\lambda^2-1)}\,z_1^3-\gamma\varepsilon z_1+\beta\Big)\,\vert\,\delta,\,\gamma
\in\mathbb{C}^*,\,\beta,\,\varepsilon\in\mathbb{C}\Big\}
\]
The inclusion 
$\mathrm{G}_\varsigma\cap\mathrm{Cent}(\phi,\mathrm{Aut}(\mathbb{C}^3)_{\mathrm{c}(\omega)})\subsetneq\mathrm{G}_\varsigma$ 
is strict; indeed
\[
\mathrm{G}_\varsigma\cap\mathrm{Cent}(\phi,\mathrm{Aut}(\mathbb{C}^3)_{\mathrm{c}(\omega)})
=\Big\{\Big(\gamma^2 z_0+\varepsilon,\gamma z_1,\gamma^3 z_2-\gamma\varepsilon z_1+
\frac{\gamma^3-1}{\lambda-1}\delta\Big)\,\vert\,\gamma\in\mathbb{C}^*,\,
\varepsilon\in\mathbb{C}\Big\}.
\]
\end{eg}

\medskip

\subsection{Non-simplicity, Tits alternative}

\smallskip

Let us recall that a \textbf{\textit{simple group}} is a non-trivial group 
$\mathrm{G}$ whose only normal subgroups are $\{\mathrm{id}\}$ and $\mathrm{G}$.

Danilov proved that $\mathrm{Aut}(\mathbb{C}^2)_\eta$ is not simple 
(\cite{Danilov}). More recently Cantat and Lamy showed 
that~$\mathrm{Bir}(\mathbb{P}^2)$ is not simple (\cite{CantatLamy}). 
As a consequence one has:

\begin{pro}\label{Pro:centre}
The groups
\[
\mathrm{Aut}(\mathbb{C}^3)_\omega,\quad\mathrm{Bir}(\mathbb{C}^3)_\omega,
\quad\mathrm{Aut}(\mathbb{C}^3)_{\mathrm{c}(\omega)},\quad[\mathrm{Aut}(\mathbb{C}^3)_{\mathrm{c}(\omega)},
\mathrm{Aut}(\mathbb{C}^3)_{\mathrm{c}(\omega)}],\quad[\mathrm{Aut}(\mathbb{C}^3)_\omega,
\mathrm{Aut}(\mathbb{C}^3)_\omega]
\]
are not simple.
\end{pro}

\begin{proof}
Since 
$[\mathrm{Aut}(\mathbb{C}^3)_{\mathrm{c}(\omega)},\mathrm{Aut}(\mathbb{C}^3)_{\mathrm{c}(\omega)}]\simeq \mathrm{Aut}(\mathbb{C}^2)_\eta$ 
and 
$[\mathrm{Aut}(\mathbb{C}^3)_\omega,\mathrm{Aut}(\mathbb{C}^3)_\omega]\simeq \mathrm{Aut}(\mathbb{C}^2)_\eta$ 
the first assertion follows from \cite{Danilov}.

\smallskip
The exact sequence ($\ref{eq:exactsequence}$) implies in particular that there 
exists a morphism with a non-trivial kernel from $\mathrm{Aut}(\mathbb{C}^3)_\omega$ 
into $\mathrm{Aut}(\mathbb{C}^2)_\eta$, hence $\mathrm{Aut}(\mathbb{C}^3)_\omega$ 
is not simple. A similar argument holds for $\mathrm{Bir}(\mathbb{C}^3)_\omega$ 
and~$\mathrm{Aut}(\mathbb{C}^3)_{\mathrm{c}(\omega)}$.
\end{proof}

The morphism
\[
\mathrm{Bir}(\mathbb{C}^3)^{\mathrm{reg}}_{\omega}\longrightarrow\mathrm{Bir}
(\mathbb{P}^2)
\]
that consists to take the restriction of 
$\phi\in\mathrm{Bir}(\mathbb{C}^3)^{\mathrm{reg}}_{\omega}$ to 
$\mathcal{H}_\infty$ has a non-trivial kernel; indeed
\[
\phi=\left(z_0-\left(\frac{P(z_1)}{Q(z_1)}\right)',z_1,z_2+\frac{P(z_1)}{Q(z_1)}
\right)
\]
with $P$, $Q$ two polynomials of degree $p$, $q$ such that $p<q+1$, is regular 
and induces the identity on $\mathcal{H}_\infty$. In particular one gets the 
following statement:

\begin{pro}\label{pro:regnotsimple}
The group $\mathrm{Bir}(\mathbb{C}^3)^{\mathrm{reg}}_{\omega}$ is not simple.
\end{pro}

Let us consider the maps 
$\psi=\Big(\gamma z_0^2z_1,\frac{1}{\gamma z_0},z_2+z_0z_1\Big)$ and 
$\phi=\left(z_0+\frac{1}{z_1^3},z_1,z_2+\frac{1}{2z_1^2}\right)$. One can check 
that~$\psi$ belongs to 
$\mathrm{Bir}(\mathbb{C}^3)_{\omega}\smallsetminus\mathrm{Bir}(\mathbb{C}^3)^{\mathrm{reg}}_{\omega}$ 
whereas $\phi$ is in $\mathrm{Bir}(\mathbb{C}^3)^{\mathrm{reg}}_{\omega}$. 
A direct computation shows that $\psi^{-1}\phi\psi$ blows down $\mathcal{H}_\infty$ 
onto $\mathbf{e}_3$. Hence one can state:

\begin{pro}\label{pro:regnotnormal}
The subgroup $\mathrm{Bir}(\mathbb{C}^3)^{\mathrm{reg}}_{\omega}$ of 
$\mathrm{Bir}(\mathbb{C}^3)_{\omega}$ is not normal.
\end{pro}

\smallskip

We will end this section by establishing Tits Alternative for 
$\mathrm{Aut}(\mathbb{C}^3)_\omega$, 
$\mathrm{Aut}(\mathbb{C}^3)_{\mathrm{c}(\omega)}$ and 
$\mathrm{Bir}(\mathbb{C}^3)_\omega$.

The derived series of a group $\mathrm{G}$ is defined as follows
\[
D_0(\mathrm{G})=\mathrm{G},\quad D_1(\mathrm{G})=[\mathrm{G},\mathrm{G}],\quad
\ldots,\quad D_{n+1}(\mathrm{G})=[\mathrm{G},D_n(\mathrm{G})]
\]
The group $\mathrm{G}$ is \textbf{\textit{solvable}} if there exists an integer 
$k$ such that $D_k(\mathrm{G})=\{\mathrm{id}\}$. The least $\ell$ such that 
$D_\ell=\{\mathrm{id}\}$ is called the \textbf{\textit{derived length}} of 
$\mathrm{G}$.

\smallskip

A group $\mathrm{G}$ satisfies the \textbf{\textit{Tits alternative}} 
if any finitely generated subgroup of $\mathrm{G}$ contains either a non-abelian 
free group, or a solvable subgroup of finite index. This alternative has been 
established by Tits for linear groups $\mathrm{GL}(n;\Bbbk)$ for any 
field $\Bbbk$~(\cite{Tits}). Lamy proves that the group of 
polynomial automorphisms of $\mathrm{Aut}(\mathbb{C}^2)$ satisfies the 
Tits alternative (\cite{Lamy}), so does Cantat for the group 
of birational maps of a complex, compact, k\"{a}hler surface (\emph{see} 
\cite{Cantat:tits}). Note that the automorphisms groups of complex, compact, 
k\"{a}hler manifolds of any dimension also satisfies Tits alternative 
(\cite{Cantat:tits, Oguiso}).

\begin{thm}\label{thm:tits}
The groups $\mathrm{Aut}(\mathbb{C}^3)_\omega$, 
$\mathrm{Aut}(\mathbb{C}^3)_{\mathrm{c}(\omega)}$ and 
$\mathrm{Bir}(\mathbb{C}^3)_\omega$ satisfy the Tits 
alternative.
\end{thm}

\begin{proof}
Let $\mathrm{G}$ be a finitely generated subgroup of 
$\mathrm{Bir}(\mathbb{C}^3)_\omega$. Set
\[
\mathrm{G}_0=\varsigma(\mathrm{G})\subset\mathrm{Bir}(\mathbb{C}^2)_\eta
\]
Since $\mathrm{Bir}(\mathbb{C}^2)_\eta$ is a subgroup of 
$\mathrm{Bir}(\mathbb{P}^2)$ that satisfies the Tits 
alternative, either $\mathrm{G}_0$ contains a non-abelian free group, or a 
solvable subgroup of finite index.

\smallskip

Assume first that $\mathrm{G}_0$ contains two elements $f$ and $h$ such that 
$\langle f,\, h\rangle\simeq \mathbb{Z}\ast\mathbb{Z}$. Let us denote by $F$, 
resp.~$H$ a lift of $f$, resp. $h$ in $\mathrm{Bir}(\mathbb{P}^3)$. 
Suppose that there exists a non-trivial word $M$ such that $M(F,H)=\{\mathrm{id}\}$. 
As~$\varsigma$ is a morphism, one gets that $M(f,h)=\{\mathrm{id}\}$: contradiction.

\smallskip

Suppose now that up to finite index $\mathrm{G}_0$ is solvable, and let $\ell$ 
be its derived length; in particular $D_\ell(\mathrm{G}_0)=~\{\mathrm{id}\}$ 
and $D_\ell(\mathrm{G})$ belongs to $\ker\varsigma$. Since
\[
\ker\varsigma=\{(z_0,z_1,z_2+\beta)\,\vert\, \beta\in\mathbb{C}\}
\]
one gets $D_{\ell+1}(\mathrm{G})=\{\mathrm{id}\}$.
\end{proof}

\medskip

\subsection{Non-conjugate isomorphic groups}

Let us denote by $\upsilon_1$ the trivial 
embedding from $(\mathrm{Aut}(\mathbb{C}^2)_\eta\vert 0)$ into $\mathrm{Aut}(\mathbb{C}^3)$
\[
\upsilon_1\colon(\mathrm{Aut}(\mathbb{C}^2)_\eta\vert 0)\hookrightarrow  \mathrm{Aut}
(\mathbb{C}^3),\qquad (\phi_0,\phi_1)\mapsto(\phi_0,\phi_1,z_2)
\]
and by $\upsilon_2$ the trivial embedding from $\mathrm{Bir}(\mathbb{P}^2)$ 
into $\mathrm{Bir}(\mathbb{P}^3)$
\[
\upsilon_2\colon\mathrm{Bir}(\mathbb{P}^2)\hookrightarrow \mathrm{Bir}
(\mathbb{P}^3),\qquad (\phi_1,\phi_2)\mapsto(z_0,\phi_1,\phi_2).
\]

Despite $\mathrm{im}\,\upsilon_1$ (resp. $\mathrm{im}\,\upsilon_2$) is isomorphic 
to $\mathrm{im}\,\varsigma$ (resp. $\mathrm{im}\,\mathcal{K}$) one has the 
following statement:

\begin{pro}
The image of $\upsilon_1$ $($resp. $\upsilon_2)$ is not $\mathrm{Aut}(\mathbb{C}^3)$-conjugate 
$($resp. $\mathrm{Bir}(\mathbb{P}^3)$-conjugate$)$ to a subgroup of 
$\mathrm{Aut}(\mathbb{C}^3)_{\mathrm{c}(\omega)}$ $($resp. 
$\mathrm{Bir}(\mathbb{C}^3)_{\mathrm{c}(\omega)})$.
\end{pro}

\begin{proof}
Let us assume that there exists $\psi$ in $\mathrm{Aut}(\mathbb{C}^3)$ (resp. 
$\mathrm{Bir}(\mathbb{P}^3)$) such that for any $\phi=(\phi_0,\phi_1)$ (resp. 
$\phi=(\phi_1,\phi_2)$) in $\mathrm{Aut}(\mathbb{C}^2)$ (resp. 
$\mathrm{Bir}(\mathbb{P}^2)$) the map $\psi\upsilon_1(\phi)\psi^{-1}$ (resp. 
$\psi\upsilon_2(\phi)\psi^{-1}$) is a contact polynomial automorphism (resp. 
contact birational map); as a result $\upsilon_1(\phi)$ (resp. 
$\upsilon_2(\phi)$) preserves a polynomial form 
$\Theta=A\mathrm{d}z_0+B\mathrm{d}z_1+C\mathrm{d}z_2$.
Looking at the restriction to any hyperplane $z_2=\lambda$ (resp. $z_0=\lambda$) 
for $\lambda$ generic one gets that all the $\phi$ preserve the foliation given by 
$\Theta_{\vert z_2=\lambda}$ (resp. $\Theta_{\vert z_0=\lambda}$): contradiction.
\end{proof}

\section{Appendix: Automorphisms group of $\mathrm{Aut}(\mathbb{C}^2)_\eta$}\label{part:appendix}

\medskip

As we recalled $\mathrm{Aut}(\mathbb{C}^2)$ is generated by $\mathrm{J}_2$
and $\mathrm{Aff}_2$. 
More precisely $\mathrm{Aut}(\mathbb{C}^2)$ has a structure of amalgamated product 
(\cite{Jung})
\[
\mathrm{Aut}(\mathbb{C}^2)=\mathrm{J}_2\ast_{\mathrm{J}_2\cap\mathrm{Aff}_2}\mathrm{Aff}_2;
\]
this is also the case for $\mathrm{Aut}(\mathbb{C}^2)_\eta$ 
(\cite[Proposition 9]{FurterLamy})
\[
\mathrm{Aut}(\mathbb{C}^2)_\eta=(\mathrm{J}_2)_\eta\ast_{(\mathrm{J}_2)_\eta\cap
(\mathrm{Aff}_2)_\eta}(\mathrm{Aff}_2)_\eta
\]

Following \cite{Deserti:autaut} we prove that:

\begin{thm}\label{thm:autauteta}
The group $\mathrm{Aut}(\mathrm{Aut}(\mathbb{C}^2)_\eta)$ is generated by 
the automorphisms of the field $\mathbb{C}$ and the group of 
$\mathrm{Aut}(\mathbb{C}^2)$-inner automorphisms.
\end{thm}

\begin{proof}[Idea of the Proof]
Let us set $\mathcal{G}=\mathrm{Aut}(\mathbb{C}^2)_\eta$. One can follow 
\cite{Deserti:autaut} and prove that if $\varphi$ is an automorphism of 
$\mathcal{G}$, then
\smallskip
\begin{itemize}
\item $\varphi((\mathrm{J}_2)_\eta)=(\mathrm{J}_2)_\eta$ up to conjugacy by an element 
of $\mathrm{Aut}(\mathbb{C}^2)$ (\cite[Proposition 4.4]{Deserti:autaut});
\smallskip
\item for any integer $k$ if 
$\mathcal{R}=\cup_{n\leq k}\langle \left(\beta z_0,\frac{z_1}{\beta}\right)\,\vert\,\beta \text{ $n$-th root of unity}\rangle$, 
then there exists $\psi$ in $(\mathrm{J}_2)_\eta$ such 
that~$\varphi(\mathcal{R})=\psi\mathcal{R}\psi^{-1}$. So one can suppose that 
$\varphi((\mathrm{J}_2)_\eta)=(\mathrm{J}_2)_\eta$ and $\varphi(\mathcal{R})=\mathcal{R}$ 
(\emph{see} \cite[Proposition~4.4]{Deserti:autaut});
\smallskip
\item set 
$\mathrm{D}_\eta=\big\{(\beta z_0,z_1/\beta)\,\vert\,\beta\in\mathbb{C}^*\big\}$ 
one can show that conjugating $\phi$ by an element of $(\mathrm{J}_2)_\eta$ one has 
$\varphi((\mathrm{J}_2)_\eta)=(\mathrm{J}_2)_\eta$ and $\varphi(\mathrm{D}_\eta)=\mathrm{D}_\eta$.
\smallskip
\item set
\[
\mathrm{T}_1=\big\{(z_0+\beta,z_1)\,\vert\,\beta\in\mathbb{C}\big\},\qquad 
\mathrm{T}_2=\big\{(z_0,z_1+\beta)\,\vert\,\beta\in\mathbb{C}\big\}
\]
and
\[
\mathrm{T}=\big\{(z_0+\gamma,z_1+\beta)\,\vert\,\gamma,\,\beta\in\mathbb{C}\big\}
\]
Since $\mathrm{T}_1\subset[[(\mathrm{J}_2)_\eta,(\mathrm{J}_2)_\eta],[(\mathrm{J}_2)_\eta,(\mathrm{J}_2)_\eta]]$, then 
$\mathrm{T}_1\subset\{(z_0+P(z_1),z_1)\,\vert\, P\in\mathbb{C}[z_1]\}$. 
As
\[
\forall\,n\in\mathbb{N},\,\forall\,\beta\in\mathbb{C}\qquad\left(\frac{z_0}{n},
nz_1\right)(z_0+\beta,z_1)^n\left(nz_0,\frac{z_1}{n}\right)=(z_0+\beta,z_1)
\]
and $\varphi(\mathrm{D}_\eta)=\mathrm{D}_\eta$, one gets
\[
\forall\,n\in\mathbb{N},\,\forall\,\beta\in\mathbb{C}\qquad\varphi\left(
\frac{z_0}{n},nz_1\right)\varphi(z_0+\beta,z_1)^n\varphi\left(nz_0,\frac{z_1}{n}
\right)=\varphi(z_0+\beta,z_1)
\]
that is
\[
\forall\,n\in\mathbb{N}\qquad\left(\frac{z_0}{\delta},\delta z_1\right)
(z_0+nP(z_1),z_1)^n\left(\delta z_0,\frac{z_1}{\delta}\right)=(z_0+P(z),z_1)
\]
so $P(z_1)=\frac{n}{\delta}P\left(\frac{z_1}{\delta}\right)$. The polynomial $P$ 
is non-zero hence $n=\delta$ and $P$ is a constant. Therefore 
$\varphi(\mathrm{T}_1)\subset~\mathrm{T}_1$.

\smallskip

The groups $\mathrm{T}_1$ and $\mathrm{T}_2$ commute, that's why
\[
\varphi(\mathrm{T}_2)\subset\big\{(z_0+P(z_1),z_1+\beta)\,\vert\, P\in
\mathbb{C}[z_1],\,\beta\in\mathbb{C}\big\}
\]
The relation
\[
\left(\frac{z_0}{n},nz_1\right)(z_0,z_1+\beta)\left(nz_0,\frac{z_1}{n}
\right)=(z_0,z_1+\beta)^n
\]
true for any integer $n$ and for any $\beta$ in $\mathbb{C}$ implies that 
$\varphi(\mathrm{T}_2)\subset\mathrm{T}_2$. The group $\mathrm{T}$ being a 
maximal abelian subgroup of $\mathcal{G}$, one has 
$\varphi(\mathrm{T})=\mathrm{T}$ and $\varphi(\mathrm{T}_i)=\mathrm{T}_i$.
\smallskip
\item There exist $\xi_1$, $\xi_2$ two additive morphisms and $\zeta$ a 
multiplicative one such that
\[
\varphi(z_0+\gamma,z_1+\beta)=(z_0+\xi_1(\gamma),z_1+\xi_2(\beta))\qquad\&
\qquad \varphi\left(\gamma z_0,\frac{z_1}{\gamma}\right)=\left(\zeta(\gamma)
z_0,\frac{z_1}{\zeta(\gamma)}\right)
\]
The statement follows from \cite[Proposition 1.4]{Deserti:autaut}.
\end{itemize}
\end{proof}

\end{document}